\documentclass[a4paper,12pt]{amsart}
\usepackage{graphicx}
\usepackage{amsmath,amssymb,amsthm,amsfonts}
\usepackage{amssymb}
\usepackage[active]{srcltx}

\newtheorem{thm}{Theorem}[section]

\newtheorem{lem}[thm]{Lemma}

\newtheorem{rem}[thm]{Remark}

\newcommand{\bremark}{\begin{rem} \textup}
\newcommand{\eremark}{\end{rem} }

\newcommand{\cuad}{{\sqcap\kern-.68em\sqcup}}

\newcommand{\R}{{\mathbb{R}}}

\renewcommand{\rho}{\varrho}
\renewcommand{\theta}{\vartheta}

\newcommand{\vp }{\varphi }

\begin{document}


\parindent 0pc
\parskip 6pt
\overfullrule=0pt

\title[Symmetry and monotonicity properties of singular solutions]{Symmetry and monotonicity properties of singular solutions to some cooperative semilinear elliptic systems involving critical nonlinearities}

\author {Francesco Esposito$^{*,+}$}

\date{\today}

\date{\today}

\address{* Dipartimento di Matematica e Informatica, UNICAL,
Ponte Pietro  Bucci 31B, 87036 Arcavacata di Rende, Cosenza, Italy.}

\address{+ Universit\'e de Picardie Jules Verne, LAMFA, CNRS UMR
7352, 33, rue Saint-Leu 80039 Amines, France}

 \email{esposito@mat.unical.it}


\thanks{F. Esposito is partially supported by PRIN project 2019, {\em Qualitative and quantitative properties of nonlinear PDEs} and by INDAM-GNAMPA project 2018 {\em Problemi ellittici semilineari: alcune idee variazionali.}}

\maketitle

\date{\today}

\begin{abstract}
We investigate qualitative properties of positive singular solutions of some elliptic systems in bounded and unbounded domains. We deduce symmetry and monotonicity properties via the moving plane procedure. Moreover, in the unbounded case, we study some cooperative elliptic systems involving critical nonlinearities in $\R^n$.
\end{abstract}

\section{introduction}
The aim of this paper is to investigate symmetry and monotonicity
properties of singular solutions to some semilinear
elliptic systems. In the first part of the paper we start by considering the
following semilinear elliptic system
\begin{equation} \label{E:systemsemilinear}
\begin{cases}
-\Delta u_i\,=f_i(u_1,\dots,u_m)& \text{in}\quad\Omega\setminus \Gamma  \\
u_i> 0 &  \text{in}\quad\Omega\setminus \Gamma  \\
u_i=0 & \text{on}\quad\partial \Omega\,
\end{cases}
\end{equation}
where $\Omega$ is a bounded smooth domain of $\mathbb{R}^n$ with
$n\geq 2$ and $i=1,...,m$ ($m \geq 2$). The technique which is mostly used in this paper is the well-known moving plane method which goes back to the seminal works of Alexandrov \cite{A} and Serrin \cite{serrin}. See also the celebrated
papers of Berestycki-Nirenberg \cite{BN} and Gidas-Ni-Nirenberg \cite{GNN}. Such a technique can be performed in general domains providing partial monotonicity results
near the boundary and symmetry when the domain is convex and
symmetric. For simplicity of exposition we assume directly in all
the paper that $\Omega$ is a convex domain which is symmetric with
respect to the hyperplane $\{x_1=0\}$. The solution has a possible
singularity on the critical set $\Gamma\subset \Omega$. When $m=1$, system \eqref{E:semilinearSym} reduces to a scalar equations that was already studied in \cite{EFS,Dino}. The moving plane procedure for semilinear elliptic systems has been firstly adapted by Troy in \cite{troy} where he considered the cooperative system \eqref{E:systemsemilinear} with $\Gamma= \emptyset$ (see also \cite{defig1,defig2,RZ}). This technique was also adapted in the case of cooperative semilinear systems in the half space by Dancer in \cite{Dan} and in the whole space by Busca and Sirakov in \cite{busca}. For the case of quasilinear elliptic systems in bounded domains we suggest \cite{MSS}. 

Moreover, motivated by \cite{leoni}, through all the paper, we assume that the following hypotheses (denoted by $(h_{f_i})$ in the sequel) hold:

\begin{itemize}
\item[$(h_{f_i})$]
\begin{itemize}
\item[(i)]  $f_i: \R^m_+ \rightarrow \R$ are assumed to be $\mathcal{C}^1$ 
functions for every $i=1,...,m$.

\item[(ii)] The functions $f_i$ ($1 \leq i \leq m$) are assumed  to satisfy the monotonicity (also known as \textit{cooperative}) conditions
$$\frac{\partial f_i}{\partial t_j} (t_1,...,t_j,...,t_m) \geq 0 \quad \text{for} \quad  i \neq j, \; 1 \leq i,j \leq m.$$
\end{itemize}
\end{itemize}

In this paper the case of singular nonlinearities for systems is not included, while it was considered in the case of scalar equations, see \cite{EFS}; about these problems we have also to mention the pioneering work of Crandall, Rabinowitz and Tartar \cite{crandall} and also  \cite{boccardo,CES,ES,lazer,stuart} for the scalar case. It would be interesting to consider in future projects a more general class of nonlinearities. In particular it would be interesting to study problems involving singular nonlinearities as in the scalar case, using some techniques developed in \cite{CES,ES}. Since we want to consider singular solutions, the natural assumption in our paper is
$$u_i \in H_{loc}^1(\Omega \setminus \Gamma) \cap
C(\overline{\Omega} \setminus \Gamma) \quad \forall i=1,...,m$$
and thus the system is
understood in the following sense:
\begin{equation} \label{E:weakfor}
\int_\Omega \nabla u_i \nabla \vp_i \, dx = \int_\Omega f_i(u_1,u_2,...,u_m) \vp_i \, dx \qquad \forall \varphi_i \in C^1_c(\Omega \setminus \Gamma)
\end{equation}
for every $i=1,...,m$.

\begin{rem}\label{E:standellest}
Note that, by the assumption $(h_{f_i})$, the right hand side in the system \eqref{E:semilinearSym} is locally bounded.	Therefore, by standard elliptic regularity theory, it follows that
$$u_i \in C_{loc}^{1,\alpha} (\Omega \setminus \Gamma),$$
\noindent where $0 < \alpha < 1$. We just remark that, in 1968, E. De Giorgi provided a counterexample showing that the scalar case is special and the regularity theory does not work in general for elliptic systems, see \cite{degiorgi}, however, in the case of equations involving Laplace operator, Schauder theory is still applicable.
\end{rem} 

Under the previous assumptions we can prove the following result:

\begin{thm} \label{E:semilinearSym}
Let $\Omega$ be a convex domain which is symmetric with respect to
the hyperplane $\{x_1=0\}$ and let $(u_1,...,u_m)$ be a solution to
\eqref{E:systemsemilinear}, where $u_i \in
H^1_{loc}(\Omega\setminus\Gamma)\cap
C(\overline\Omega\setminus\Gamma)$ for every $i=1,...,m$. Assume that each $f_i$ fulfills $(h_{f_i})$. Assume also that
$\Gamma$ is a point if $n=2$ while
 $\Gamma$ is closed and such that
$$\underset{\R^n}{\operatorname{Cap}_2}(\Gamma)=0,$$
if $n\geq 3$. \noindent Then, if $\Gamma\subset\{x_1=0\}$, it
follows that $u_i$ is symmetric with respect to the
hyperplane $\{x_1=0\}$ and increasing in the $x_1$-direction in
$\Omega\cap\{x_1<0\}$, for every $i=1,...,m$. Furthermore
\[\partial_{x_1} u_i>0 \qquad  \text{in}\quad \Omega\cap\{x_1<0\}\, ,
\]
for every $i=1,...,m$.
\end{thm}

The technique developed in the first part of the paper and in \cite{EFS, EMS, Dino} (see also \cite{mps} for the nonlocal setting) is very powerful and can be adapted to some cooperative systems in $\R^n$ involving critical nonlinearities. Papers on existence or qualitative properties of solutions to systems with critical growth in $\R^n$ are very few, due to the lack of compactness given by the Talenti bubbles and the difficulties arising from the lack of good variational methods. We refer the reader to \cite{busca,pistoia,grossi1,grossi2,guoliu,Wang} for this kind of systems. The starting point of the second part of the paper is the study of qualitative properties of singular solutions to the following $m \times m$ system of equations

\begin{equation}\label{E:criticalsingularequbdd}
\begin{cases}
\displaystyle -\Delta u_i = \sum_{j=1}^m a_{ij} u_j^{2^*-1}  &  \text{in}\quad\R^n\setminus \Gamma,\\
u_i> 0 &  \text{in}\quad\R^n\setminus \Gamma,
\end{cases}
\end{equation}

where $i=1,...,m$, $m \geq 2$, $n \geq 3$ and the matrix $A:=(a_{ij})_{i,j=1,...,m}$ is symmetric and such that 
\begin{equation}\label{E:condition}
\sum_{j=1}^m a_{ij}=1 \; \text{for every $i=1,...,m$.}
\end{equation}
These kind of systems, with $\Gamma = \emptyset$, was studied by Mitidieri in \cite{mitidieri, mitidieri1} considering the case $m=2$, $A=\begin{pmatrix} 0 & 1 \\ 1 & 0 \end{pmatrix}$ and it is known in the literature as nonlinearity belonging
to the {\em critical hyperbola}.

If $m=1$, then \eqref{E:criticalsingularequbdd} reduces to the classical critical Sobolev equation

\begin{equation}
\label{E:problemubdd}
\begin{cases}
-\Delta u\, = u^{2^{*}-1} & \text{in}\quad \mathbb{R}^{n}\setminus
\Gamma
\\
u> 0 & \text{in}\quad \mathbb{R}^{n}\setminus \Gamma .
\\
\end{cases}
\end{equation}

that can be found in \cite{EFS, Dino}. If $\Gamma$ reduces to a single point we find the result contained in \cite{terracini}, while if $\Gamma = \emptyset$ then system \eqref{E:problemubdd} reduces to the classical Sobolev equation (see \cite{CGS}). For existence results of radial and nonradial solutions for \eqref{E:criticalsingularequbdd}, we refer to some interesting papers \cite{grossi1, grossi2}. We want to remark that in \cite{grossi1, grossi2} the authors treat the general case of a matrix $A$ in which its entries $a_{ij}$ are not necessarily positive and this fact implies that it is not possible to apply the maximum principle. As remarked above the natural assumption is
$$u_i \in H_{loc}^1(\R^n \setminus \Gamma)  \quad \forall i=1,...,m$$
and thus the system is
understood in the following sense:
\begin{equation} \label{E:weakforubdd}
\int_{\R^n} \nabla u_i \nabla \vp_i \, dx = \sum_{j=1}^m a_{ij} \int_{\R^n} u_j^{2^*-1} \vp_i \, dx \qquad \forall \varphi_i \in C^1_c(\R^n \setminus \Gamma)
\end{equation}
for every $i=1,...,m$. 

What we are going to show is the following result:

\begin{thm}\label{E:main2}
	Let $n\geq 3$ and let $(u_1,...,u_m)$ be a solution to \eqref{E:criticalsingularequbdd}, where  $u_i \in H^1_{loc}(\mathbb{R}^n\setminus\Gamma)$ for every $i=1,...,m$. Assume that the matrix $A=(a_{ij})_{i,j=1,...,m}$, defined above, is symmetric,  $a_{ij} \geq 0$ for every $i,j=1,...,m$ and it satisfies \eqref{E:condition}. Moreover at least one of the $u_i$
	has a non-removable\footnote{Here we mean that the solution $(u_1,...,u_m)$ does not admit a smooth extension all over the whole space. Namely it is not possible to find $\tilde u_i \in H^1_{loc}(\mathbb{R}^n)$ with $u_i\equiv \tilde u_i$  in $\mathbb{R}^n\setminus\Gamma$, for some $i=1,...,m$.} singularity in the singular set $\Gamma$, where
	$\Gamma$ is a closed and proper subset of $\{x_1=0\}$ such that
	$$\underset{\R^n}{\operatorname{Cap}_2}(\Gamma)=0. $$
	\noindent Then, all the $u_i$ are symmetric with respect to the hyperplane $\{x_1=0\}$. The same conclusion is true if $\{x_1=0\}$ is
	replaced by any affine hyperplane. If at least one of the  $u_i$ has only a non-removable singularity at the origin for every $i=1,...,m$, then each $u_i$ is radially symmetric about the origin and radially decreasing.
\end{thm}

Another interesting elliptic system involving Sobolev critical exponents is the following one:

\begin{equation}\label{E:doublecriticalsingularequbdd}
\begin{cases}
\displaystyle -\Delta u\,=u^{2^*-1}+ \frac{\alpha}{2^*} u^{\alpha-1} v^\beta & \text{in}\quad\R^n\setminus \Gamma\\
\displaystyle -\Delta v\,=v^{2^*-1}+ \frac{\beta}{2^*} u^\alpha v^{\beta-1} &
\text{in}\quad\R^n\setminus \Gamma \\
u,v > 0 &  \text{in}\quad\R^n\setminus \Gamma,
\end{cases}
\end{equation}

where $\alpha, \beta > 1$, $\alpha+\beta=2^*:=\frac{2n}{n-2} \ (n \geq 3)$ 

The solutions to \eqref{E:doublecriticalsingularequbdd} are solitary waves for a system of coupled Gross--Pitaevskii equations. This type of systems arises, e.g., in the Hartree--Fock theory for double condensates, that is, Bose-Einstein condensates of two different hyperfine states which overlap in space. Existence results for these kind of systems are very complicated and the existence of nontrivial solutions is deeply related to the parameters $\alpha, \beta$ and $n$. System \eqref{E:doublecriticalsingularequbdd} with $\Gamma= \emptyset$ was studied in \cite{ambrosetti,bartsch1,bartsch2,Wang,sirakov,nicola2}. In particular in \cite{Wang} the authors show a uniqueness result, for least energy solutions, under suitable assumptions on the parameters $\alpha, \beta$ and $n$, while in \cite{pistoia} the authors study also the competitive setting, showing that the system admits infinitely many fully nontrivial solutions, which are not conformally equivalent. Motivated by their physical applications, weakly coupled elliptic systems have received
much attention in recent years, and there are many results for the cubic case where $\Gamma= \emptyset$, $\alpha = \beta =2$ and  $2^*$ is replaced by $4$ in low dimensions $n=3,4$, see e.g. \cite{ambrosetti, bartsch1, bartsch2, lin1, lin2, nicola1, nicola2}. Since our technique does not work  when $1 < \alpha < 2$ or $1< \beta <2$, here we study the case $\alpha, \beta \geq 2$ and $n=3$ or $n=4$, since we are assuming that $\alpha + \beta = 2^*$.

\begin{thm}\label{E:main3}
	Let $n=3$ or $n=4$ and let $(u,v) \in H^1_{loc}(\mathbb{R}^n\setminus\Gamma) \times H^1_{loc}(\mathbb{R}^n\setminus\Gamma)$
	be a solution to \eqref{E:doublecriticalsingularequbdd}. Assume that the solution $(u,v)$ has a non-removable\footnote{As above, we mean that the solution $(u,v)$ does not admit a smooth extension all over the whole space. Namely it is not possible to find $(\tilde u, \tilde v) \in H^1_{loc}(\mathbb{R}^n) \times H^1_{loc}(\mathbb{R}^n)$ with $u\equiv \tilde u$ or $v \equiv \tilde v$ in $\mathbb{R}^n\setminus\Gamma$.} singularity in the singular set $\Gamma$, where $\Gamma$ is a closed and proper subset of $\{x_1=0\}$ such that
	$$\underset{\R^n}{\operatorname{Cap}_2}(\Gamma)=0. $$
	Moreover let us assume that $\alpha, \beta \geq 2$ and that holds $\alpha + \beta = 2^*$. Then, $u$ and $v$ are symmetric with respect to the hyperplane $\{x_1=0\}$. The same conclusion is true if $\{x_1=0\}$ is
	replaced by any affine hyperplane. If at least one between $u$ and $v$ has only a non-removable singularity at the origin, then $(u,v)$ is radially symmetric about the origin and radially decreasing.
\end{thm}

\noindent When the paper was completed we learned that the case of bounded domains was also considered in \cite{BVV} (see \cite{BVV0}), obtaining similar results.

\section{Notations and preliminary results} \label{E:notations}

We need to fix some notations. For a real number $\lambda$ we set
\begin{equation}\label{E:eq:sn2} \nonumber
\Omega_\lambda=\{x\in \Omega:x_1 <\lambda\}
\end{equation}
\begin{equation}\label{E:eq:sn3} \nonumber
x_\lambda= R_\lambda(x)=(2\lambda-x_1,x_2,\ldots,x_n)
\end{equation}
which is the reflection through the hyperplane $T_\lambda :=\{ x_1=
\lambda\}$. Also let
\begin{equation}\label{E:eq:sn4} \nonumber
a=\inf _{x\in\Omega}x_1.
\end{equation}

Since $\Gamma$ is compact and of zero capacity, $u_i$ is
defined a.e. on $\Omega$ and Lebesgue measurable on $\Omega$ for every $i=1,...,m$.
Therefore the functions \begin{equation}\label{E:eq:sn33} \nonumber  
u_{i,\lambda} := u_i \circ R_{\lambda}
\end{equation}
are Lebesgue measurable on $R_{\lambda}(\Omega)$. Similarly, $\nabla
u_i$ and $\nabla u_{i,\lambda}$ are Lebesgue
measurable on $\Omega$ and $R_{\lambda}(\Omega)$ respectively.


In the same spirit of \cite{EFS} we recall some useful properties of the $2$-capacity. It is easy to see that, if
$\underset{\R^n}{\operatorname{Cap}_2}(\Gamma)=0$, then
$\underset{\R^n}{\operatorname{Cap}_2}(R_{\lambda}(\Gamma))=0$.
Another consequence of our assumptions
is that
$\underset{B^{\lambda}_{\epsilon}}{\operatorname{Cap}_2}(R_{\lambda}(\Gamma))=0$
for any open neighborhood $\mathcal{B}^{\lambda}_{\epsilon}$ of
$R_{\lambda}(\Gamma)$. Indeed, recalling that $\Gamma$ is a point if
$n=2$ while
 $\Gamma$ is closed with
$\underset{\R^n}{\operatorname{Cap}_2}(\Gamma)=0$ if $n\geq 3$ by
assumption, it follows that
$$\underset{\mathcal{B}^{\lambda}_{\epsilon}}{\operatorname{Cap}_2}(R_{\lambda}(\Gamma)):=
\inf \left\{ \int_{\mathcal{B}^{\lambda}_{\epsilon}} |\nabla \varphi
|^2 dx < + \infty \; : \; \varphi \geq 1 \ \text{in} \
\mathcal{B}^{\lambda}_{\delta}, \; \varphi \in \
C^{\infty}_c(\mathcal{B}^{\lambda}_{\epsilon}) \right\}=0,$$
\noindent for some neighborhood $\mathcal{B}^{\lambda}_{\delta}
\subset \mathcal{B}^{\lambda}_{\varepsilon}$ of
$R_{\lambda}(\Gamma)$. From this, it follows that there exists
$\varphi_{\varepsilon} \in \
C^{\infty}_c(\mathcal{B}^{\lambda}_{\epsilon})$ such that
$\varphi_{\varepsilon} \geq 1$ in $\mathcal{B}^{\lambda}_{\delta}$
and $\displaystyle \int_{\mathcal{B}^{\lambda}_{\epsilon}} |\nabla
\varphi_{\varepsilon} |^2 dx < \varepsilon$.

Now we construct a function $\psi_{\varepsilon} \in C^{0,1}(\R^n,
[0,1])$ such that $\psi_{\varepsilon} = 1$ outside
$\mathcal{B}_{\varepsilon}^{\lambda}$, $\psi_{\varepsilon} = 0$ in
$\mathcal{B}_{\delta}^{\lambda}$ and
$$\int_{\R^n} |\nabla \psi_{\varepsilon} |^2 dx = \int_{\mathcal{B}^{\lambda}_{\epsilon}} |\nabla \psi_{\varepsilon}
|^2 dx < 4 \varepsilon.$$ \noindent To this end we consider the
following Lipschitz continuous function

$$T_1(s)= \begin{cases}
1 &  \text{if}\quad s \le 0  \\
-2s + 1 & \text{if}\quad 0 \le s \le \frac{1}{2} \\
0 &  \text{if}\quad s \ge \frac{1}{2}  \\
\end{cases}$$

and we set
\begin{equation} \label{E:test1}
\psi_{\varepsilon} := T_1 \circ \varphi_{\varepsilon}
\end{equation}
where we have extended  $\varphi_{\varepsilon}$ by zero outside
$\mathcal{B}_{\varepsilon}^{\lambda}$. Clearly $\psi_{\varepsilon}
\in C^{0,1}(\R^n), \ 0 \le \psi_{\varepsilon} \le 1 $ and

$$\int_{\mathcal{B}^{\lambda}_{\epsilon}} |\nabla \psi_{\varepsilon}
|^2 dx \leq 4 \int_{\mathcal{B}^{\lambda}_{\epsilon}} |\nabla
\varphi_{\varepsilon} |^2 dx  < 4 \varepsilon.$$

Now we  set $\gamma_\lambda:= \partial \Omega \cap T_{\lambda}$.
Recalling that $\Omega$ is convex, it is easy to  deduce that
$\gamma_\lambda$ is made of two points in dimension two. If, instead, 
$n\geq3$ then it follows that $\gamma_\lambda$ is a smooth manifold
of dimension $n-2$. Note in fact that locally $\partial\Omega$ is
the zero level set of a smooth function $g(\cdot)$ whose gradient is
not parallel to the $x_1$-direction since $\Omega$ is convex. Then
it is sufficient to observe that locally $ \partial \Omega \cap
T_{\lambda}\equiv\{g(\lambda,x')=0\}$ and use the \emph{implicit
function theorem} exploiting the fact that
$\nabla_{x'}g(\lambda,x')\neq0$. This implies that
$\underset{\R^n}{\operatorname{Cap}_2}(\gamma_\lambda)=0$, see  e.g.
\cite{evans}. So, as before,
$\underset{\mathcal{I}^\lambda_{\tau}}{\operatorname{Cap}_2}(\gamma_\lambda)=0$
for any open neighborhood of $\gamma_\lambda$ and then
there exists $\varphi_{\tau} \in C^{\infty}_c
(\mathcal{I}^\lambda_{\tau})$ such that $\varphi_{\tau} \geq 1$ in a
neighborhood $\mathcal{I}^\lambda_{\sigma}$ with
 $\gamma_\lambda\subset\mathcal{I}^\lambda_{\sigma} \subset
\mathcal{I}^\lambda_{\tau}$. As above, we set
\begin{equation}\label{E:test2}\phi_{\tau} :=T_1 \circ \varphi_{\tau}
\end{equation}
where we have extended  $\varphi_{\tau}$ by zero outside
$\mathcal{I}^\lambda_{\tau}$. Then, $ \phi_{\tau} \in C^{0,1}(\R^n),
0 \le \phi_{\tau} \le 1, \phi_{\tau} = 1$ outside
$\mathcal{I}^\lambda_{\tau}, \phi_{\tau} = 0$ in
$\mathcal{I}^\lambda_{\sigma}$ and

$$\int_{\R^n} |\nabla \phi_\tau
|^2 dx = \int_{\mathcal{I}^\lambda_{\tau}} |\nabla \phi_\tau |^2 dx
\leq 4  \int_{\mathcal{I}^\lambda_{\tau}} |\nabla \varphi_\tau |^2
dx  < 4 \tau.$$

\section{Proof of Theorem \ref{E:semilinearSym}}\label{E:semSym}
Let us set
$$w_{i,\lambda}^+=(u_i- u_{i,\lambda})^+$$
where $i=1,...,m$. We will prove the result by showing that, actually, it holds $w_{i,\lambda}^+ \equiv 0$ for $i=1,...,m$. 
To prove this, we have to perform the moving plane method.

In the following we will exploit the fact that $(u_{1,\lambda},...,u_{1,\lambda})$ is a solution to
\begin{equation}\label{E:weakforRef}
\int_{\Omega_\lambda} \nabla u_{i,\lambda} \nabla \vp_i \, dx = \int_{\Omega_\lambda} f_i(u_{1,\lambda},u_{2,\lambda},...,u_{m,\lambda}) \vp_i \, dx \qquad \forall \varphi_i \in C^1_c(\Omega_\lambda \setminus R_\lambda(\Gamma))
\end{equation}

for every $i=1,...,m$, where $\Omega_\lambda:=R_\lambda(\Omega)$.

We start by recalling the following helpful lemma, whose proof can be found in \cite{EFS}.

\begin{lem}[\cite{EFS}]\label{E:leaiuto} 
Let $\lambda \in (a,0)$ be such that $R_\lambda (\Gamma) \cap {\overline {\Omega}} = \emptyset$ and consider the function
$$\varphi_i\,:= \begin{cases}\, w_{i,\lambda}^+ \phi_{\tau}^2  & \text{in}\quad\Omega_\lambda, \\
0 &  \text{in}\quad \R^n\setminus \Omega_\lambda,
\end{cases}$$
where $\phi_{\tau}$ is as in \eqref{E:test2}, for $i=1,...,m$. Then, $\varphi_i \in C^{0,1}_c (\Omega) \cap C^{0,1}_c (R_{\lambda}(\Omega)), \, \varphi_i$ has compact support contained in $(\Omega \setminus \Gamma) \cap (R_{\lambda}(\Omega) \setminus R_{\lambda}(\Gamma)) \cap \{ x_1 \le \lambda \}$ and
\begin{equation}\label{E:gradvarphi}\nonumber	
\nabla \varphi_i = \phi_\tau^2 \nabla w_{i,\lambda}^+  + 2 \phi_\tau
w_{i,\lambda}^+ \nabla \phi_\tau \quad {\text {a.e. on }} \, \, \Omega \cup R_{\lambda}(\Omega),
\end{equation}
for every $i=1,...,m$. If $\lambda \in (a,0)$ is such that $R_\lambda (\Gamma) \cap {\overline {\Omega}} \neq \emptyset$, the same conclusions hold true for the function
$$\varphi_i\,:= \begin{cases}
\, w_{i,\lambda}^+ \psi_\varepsilon^2 \phi_{\tau}^2  & \text{in}\quad\Omega_\lambda, \\
0 &  \text{in}\quad \R^n \setminus \Omega_\lambda,
\end{cases}$$
where $\psi_\varepsilon$ is defined as in \eqref{E:test1} and $\phi_{\tau}$
as in \eqref{E:test2}, for every $i=1,...,m$. Furthermore, a.e. on $ \Omega \cup R_{\lambda}(\Omega)$,
\begin{equation}\label{E:gradvarphi2}
\nabla \varphi_i = \psi_\varepsilon^2 \phi_\tau^2 \nabla w_{i,\lambda}^+ +
2 w_{i,\lambda}^+(\psi_\varepsilon^2 \phi_\tau \nabla \phi_\tau + \psi_\varepsilon \phi_\tau^2  \nabla  \psi_\varepsilon).
\end{equation}
{In particular, $\varphi_i \in C^{0,1}(\overline{\Omega_\lambda})$, ${\varphi_i}_{\vert _{\partial \Omega_\lambda}} = 0$   and so $ \varphi_i \in H^1_0 (\Omega_\lambda)$, for every $i=1,...,m$.}
\end{lem}

Now we are ready to prove an essential tool that we will use to start the moving plane procedure.
\begin{lem}\label{E:leaiuto2}
Under the assumptions of Theorem \ref{E:semilinearSym}, let $a<\lambda<0$. Then $w_{i, \lambda}^+\in H^1_0(\Omega_\lambda)$ for every $i=1,...,m$ and
\begin{equation}\label{E:gradbdd}
\sum_{i=1}^m\int_{\Omega_\lambda}|\nabla w_{i,\lambda}^+|^2\,dx\leq \frac{m}{2}\sum_{i=1}^m (1+C_i^2) \|u_i\|^2_{L^\infty(\Omega_\lambda)} \vert \Omega\vert \,.
\end{equation}
{where $\vert \Omega \vert$ denotes the $n-$dimensional Lebesgue measure of $ \Omega$ and $C_i$ is a positive constant depending only on $f_i$. }
\end{lem}

\begin{proof}
For $\psi_\varepsilon$ as in \eqref{E:test1} and $\phi_{\tau}$
as in \eqref{E:test2}, we consider the functions $\varphi_i$ defined in Lemma \ref{E:leaiuto}. In view of the properties of $\varphi_i$, stated in Lemma \ref{E:leaiuto}, and a standard density argument, we can use $\varphi_i$ as test function in
\eqref{E:weakfor} and \eqref{E:weakforRef} so that, subtracting, we get

\begin{equation}\label{E:subtraction}
\begin{split}
\int_{\Omega_\lambda}|\nabla w_{i,\lambda}^+ |^2\psi_\varepsilon^2
\phi_{\tau}^2 \,dx&= -2\int_{\Omega_\lambda}\nabla w_{i,\lambda}^+ \nabla \psi_\varepsilon w_{i,\lambda}^+ \psi_\varepsilon \phi_{\tau}^2 \,dx - 2\int_{\Omega_\lambda}\nabla w_{i,\lambda}^+ \nabla \phi_\tau w_{i,\lambda}^+ \psi_\varepsilon^2 \phi_{\tau}\,dx\\
&+\int_{\Omega_\lambda} [f_i(u_1,u_2,...,u_m)-f_i(u_{1,\lambda},u_{2,\lambda},...,u_{m,\lambda})] w_{i,\lambda}^+ \psi_\varepsilon^2
\phi_{\tau}^2\, dx.\\
\end{split}
\end{equation}

Exploiting Young's inequality in the right hand side of \eqref{E:subtraction}, we get that
\begin{equation}\label{E:stima1}
\begin{split}
\int_{\Omega_\lambda}|\nabla w_{i,\lambda}^+ |^2\psi_\varepsilon^2
\phi_{\tau}^2 \,dx &\leq \frac{1}{4}\int_{\Omega_\lambda}|\nabla
w_{i,\lambda}^+|^2\psi_\varepsilon^2 \phi_{\tau}^2 \,dx
+4\int_{\Omega_\lambda}|\nabla \psi_\varepsilon|^2 (w_{i,\lambda}^+)^2
\phi_{\tau}^2
\,dx \\
&+\frac{1}{4}\int_{\Omega_\lambda}|\nabla
w_{i,\lambda}^+|^2\psi_\varepsilon^2 \phi_{\tau}^2 \,dx +
4\int_{\Omega_\lambda}|\nabla \phi_\tau|^2 (w_{i,\lambda}^+)^2
\psi_\varepsilon^2\, dx\\
&+\int_{\Omega_\lambda} [f_i(u_1,u_2,...,u_m)-f_i(u_{1,\lambda},u_{2,\lambda},...,u_{m,\lambda})] w_{i,\lambda}^+ \psi_\varepsilon^2
\phi_{\tau}^2\, dx.\\
\end{split}
\end{equation}

The last term of the right hand side of \eqref{E:stima1} can be rewritten as follows
\begin{equation}\label{E:estimateF1}
\begin{split}
&\int_{\Omega_\lambda} [f_i(u_1,u_2,...,u_m)-f_i(u_{1,\lambda},u_{2,\lambda},...,u_{m,\lambda})] w_{i,\lambda}^+ \psi_\varepsilon^2
\phi_{\tau}^2\, dx \\
=&\int_{\Omega_\lambda} [f_i(u_1,u_2,...,u_m)\pm f_i(u_{1,\lambda},u_2,...,u_m) -  f_i(u_{1,\lambda},u_{2,\lambda},...,u_{m,\lambda})] w_{i,\lambda}^+ \psi_\varepsilon^2 \phi_{\tau}^2\, dx\\
=& \int_{\Omega_\lambda} [f_i(u_1,u_2,...,u_m)- f_i(u_{1,\lambda},u_2,...,u_m) +  f_i(u_{1,\lambda},u_2,...,u_m) - f_i(u_{1,\lambda},u_{2,\lambda},u_3,...,u_m)\\ &+f_i(u_{1,\lambda},u_{2,\lambda},u_3,...,u_m) - \cdots -  f_i(u_{1,\lambda},u_{2,\lambda},...,u_{m,\lambda})] w_{i,\lambda}^+ \psi_\varepsilon^2 \phi_{\tau}^2\, dx\\
\end{split}
\end{equation}

Using the fact that $f_i$ are $\mathcal{C}^1$ functions $(h_{f_i})-(i)$ and they satisfy $(h_{f_i})-(ii)$, by \eqref{E:estimateF1} we have

\begin{equation}\label{E:estimateF2}
\int_{\Omega_\lambda} [f_i(u_1,u_2,...,u_m)-f_i(u_{1,\lambda},u_{2,\lambda},...,u_{m,\lambda})] w_{i,\lambda}^+ \psi_\varepsilon^2
\phi_{\tau}^2\, dx \leq \sum_{j=1}^m C_j (f_j)\int_{\Omega_\lambda}  w_{j,\lambda}^+ w_{i,\lambda}^+ \psi_\varepsilon^2 \phi_{\tau}^2\, dx.
\end{equation}

Now compiling all the previous estimates and exploiting Young's inequality in the right hand side of \eqref{E:estimateF2} we obtain 
\begin{equation}\label{E:stima2}
\begin{split}
\int_{\Omega_\lambda}|\nabla w_{i,\lambda}^+ |^2\psi_\varepsilon^2
\phi_{\tau}^2 \,dx &\leq 8\int_{\Omega_\lambda}|\nabla \psi_\varepsilon|^2 (w_{i,\lambda}^+)^2 \phi_{\tau}^2 \,dx + 8\int_{\Omega_\lambda}|\nabla \phi_\tau|^2 (w_{i,\lambda}^+)^2 \psi_\varepsilon^2\, dx\\
&+ m \int_{\Omega_\lambda}  (w_{i,\lambda}^+)^2 \psi_\varepsilon^2 \phi_{\tau}^2\, dx +\sum_{j=1}^m C_j^2 \int_{\Omega_\lambda}  (w_{j,\lambda}^+)^2 \psi_\varepsilon^2 \phi_{\tau}^2\, dx.\\
\end{split}
\end{equation}

By \eqref{E:stima2} summing with respect to $i$ we get

\begin{equation}\label{E:stima3} \nonumber
\begin{split}
\sum_{i=1}^m\int_{\Omega_\lambda}|\nabla w_{i,\lambda}^+ |^2\psi_\varepsilon^2
\phi_{\tau}^2 \,dx &\leq 8 \sum_{i=1}^m \left(\int_{\Omega_\lambda}|\nabla \psi_\varepsilon|^2 (w_{i,\lambda}^+)^2 \phi_{\tau}^2 \,dx + \int_{\Omega_\lambda}|\nabla \phi_\tau|^2 (w_{i,\lambda}^+)^2 \psi_\varepsilon^2\, dx\right)\\
&+\frac{m}{2}\sum_{i=1}^m (1+C_i^2) \int_{\Omega_\lambda}  (w_{i,\lambda}^+)^2 \psi_\varepsilon^2 \phi_{\tau}^2\, dx.\\
\end{split}
\end{equation}

Taking into account the properties of
$\psi_\varepsilon$ and $\phi_\tau$, we see that
\begin{equation}\label{E:bbeg} \nonumber
\int_{\Omega_\lambda}|\nabla \psi_\varepsilon|^2\,dx=
\int_{\Omega_\lambda\cap(\mathcal B_\varepsilon^\lambda\setminus
	\mathcal B_{\delta}^\lambda)}|\nabla \psi_\varepsilon|^2\,dx < 4
\varepsilon,
\end{equation}
\begin{equation}\label{E:bbeg2} \nonumber
\int_{\Omega_\lambda}|\nabla \phi_\tau|^2\,dx=
\int_{\Omega_\lambda\cap(\mathcal I^\lambda_\tau \setminus \mathcal
	I^\lambda_{\sigma})}|\nabla \phi_\tau|^2\,dx < 4 \tau,
\end{equation}
which combined with $0\leq w_{i,\lambda}^+\leq u_i$, for every $i=1,...,m$, 
immediately lead to
\begin{equation}\nonumber
\begin{split}
\sum_{i=1}^m \int_{\Omega_\lambda}|\nabla w_{i,\lambda}^+|^2\psi_\varepsilon^2
\phi_{\tau}^2 \,dx \leq 32(\varepsilon +\tau) \sum_{i=1}^m \|u_i\|^2_{L^\infty
(\Omega_\lambda)} +\frac{m}{2}\sum_{i=1}^m (1+C_i^2) \|u_i\|^2_{L^\infty
(\Omega_\lambda)} \vert \Omega\vert \,.
\end{split}
\end{equation}

By Fatou Lemma, as $\varepsilon$ and $\tau$ tend to zero, we have \eqref{E:gradbdd}.  To conclude we note that $ \varphi_i \to  w_{i,\lambda}^+$ in $L^2(\Omega)$,  as $\varepsilon$ and $\tau$ tend to zero, by definition of $ \varphi_i $ for every $i=1,...,m$. Also, $\nabla \varphi \to \nabla w_{i,\lambda}^+$ in $L^2(\Omega_{\lambda})$, by \eqref{E:gradvarphi2}. Therefore, $ w_{i,\lambda}^+$ in $H^1_0(\Omega_\lambda)$, since {$\varphi_i \in H^1_0(\Omega_\lambda)$} again by Lemma \ref{E:leaiuto}, for every $i=1,...,m$, which concludes the proof.

\end{proof}

\begin{proof}[Proof of Theorem \ref{E:semilinearSym}]
We define
\begin{equation}\nonumber
\Lambda_0=\{a<\lambda<0 : u_i\leq
u_{i,t}\,\,\,\text{in}\,\,\,\Omega_t\setminus
R_t(\Gamma)\,\,\,\text{for all $t\in(a,\lambda]$ and for every i=1,...,m.}\} 
\end{equation}
and to start with the moving plane procedure, we have to prove that
	
\emph{Step 1 :  $\Lambda_0 \neq \emptyset$}. Fix $ \lambda_0 \in (a,0)$ such that	$R_{\lambda_0}(\Gamma) \subset \Omega^c$, then for every $a< \lambda < \lambda_0$, we also have that $R_\lambda(\Gamma)\subset \Omega^c$. 
For any $ \lambda$ in this set we consider, on the domain $\Omega$, the function $\varphi_i\,:=\, w_{i,\lambda}^+ \phi_{\tau}^2 \chi_{\Omega_\lambda},$ where $\phi_{\tau}$ is as in \eqref{E:test2} and we proceed as in the proof of Lemma \ref{E:leaiuto2}. That is, by Lemma \ref{E:leaiuto} and a density argument, we can use $\varphi_i$ as test function in
\eqref{E:weakfor} and \eqref{E:weakforRef} so that, subtracting, we get
\begin{equation}\label{E:subtraction2} \nonumber
\begin{split}
\int_{\Omega_\lambda}|\nabla w_{i,\lambda}^+ |^2
\phi_{\tau}^2 \,dx =&  - 2\int_{\Omega_\lambda}\nabla w_{i,\lambda}^+ \nabla \phi_\tau w_{i,\lambda}^+ \phi_{\tau}\,dx\\
&+\int_{\Omega_\lambda} [f_i(u_1,u_2,...,u_m)-f_i(u_{1,\lambda},u_{2,\lambda},...,u_{m,\lambda})] w_{i,\lambda}^+
\phi_{\tau}^2\, dx.\\
\end{split}
\end{equation}
	
Exploiting Young's inequality and the assumption $(h_{f_i})$, then we get that
\begin{equation}\nonumber
\begin{split}
\int_{\Omega_\lambda}|\nabla w_{i,\lambda}^+|^2 \phi_{\tau}^2 \,dx &\leq
\frac{1}{2}\int_{\Omega_\lambda}|\nabla w_{i,\lambda}^+|^2 \phi_{\tau}^2
\,dx + 2\int_{\Omega_\lambda}|\nabla \phi_\tau|^2 (w_{i,\lambda}^+)^2 dx\\
&+ \sum_{i=1}^m C_j \int_{\Omega_\lambda}w_{j,\lambda}^+ w_{i,\lambda}^+ \phi_{\tau}^2\,dx.
\end{split}
\end{equation}

Taking into account the properties of $\phi_\tau$, we see that
\begin{equation}\label{E:bbeg1} \nonumber
\int_{\Omega_\lambda}|\nabla \phi_\tau|^2 (w_{i,\lambda}^+)^2 dx \leq \|
u_i \|_{L^\infty (\Omega_\lambda)}^2 \int_{\Omega_\lambda \cap (\mathcal
I^\lambda_\tau \setminus \mathcal I^\lambda_{\sigma})}|\nabla
\phi_\tau|^2\,dx\leq 4 \| u_i \|_{L^\infty(\Omega_\lambda)}^2 \cdot
\tau.
\end{equation}
We therefore deduce that
\begin{equation}\nonumber
\begin{split}
\sum_{i=1}^m \int_{\Omega_\lambda}|\nabla w_{i,\lambda}^+|^2 \phi_{\tau}^2 \,dx \leq
16 \tau \sum_{i=1}^m\|u_i\|_{L^\infty (\Omega_\lambda)} +
\frac{m}{2} \sum_{i=1}^m (1+C_i^2)
\int_{\Omega_\lambda}(w_{i,\lambda}^+)^2 \phi_{\tau}^2\,dx.
\end{split}
\end{equation}
By Fatou Lemma, as $\tau$ tends to zero, we have
\begin{equation}\label{E:ff}
\begin{split}
\sum_{i=1}^m \int_{\Omega_\lambda}|\nabla w_{i,\lambda}^+|^2 \,dx &\leq
\frac{m}{2} \sum_{i=1}^m  (1+C_i^2)
\int_{\Omega_\lambda}(w_{i,\lambda}^+)^2\,dx\\
& \leq  \frac{m}{2} \sum_{i=1}^m (1+C_i^2) C_{i,p}^2(\Omega_\lambda)
\int_{\Omega_\lambda}|\nabla w_{i,\lambda}^+|^2  \,dx,
\end{split}
\end{equation}
where $C_{i,p}(\cdot)$ is the Poincar\'e constant (in the Poincar\'e inequality in $H^1_0(\Omega_\lambda)$). Since $C_{i,p}(\Omega_\lambda) \to 0$ as $ \lambda \to a$, we can find $ \lambda_1 \in (a, \lambda_0)$, such that 
$$  C_{i,p}(\Omega_\lambda)< \frac{1}{\sqrt{m (1+C_i^2)}}\, \qquad \forall \lambda \in (a, \lambda_1) \; \text{and for every $i=1,...,m$},$$ so that by \eqref{E:ff}, we
deduce that
$$\int_{\Omega_\lambda}|\nabla w_{i,\lambda}^+|^2  \,dx \leq 0 \qquad \forall \lambda \in (a, \lambda_1) \; \text{and for every $i=1,...,m$},$$ proving
that $u_i \leq u_{i,\lambda}$ in $\Omega_\lambda \setminus R_\lambda(\Gamma)$ for $\lambda$ close to $a$, which implies the desired conclusion $\Lambda_0 \neq \emptyset$.

Now we can set
\begin{equation}\nonumber
\lambda_0=\sup\,\Lambda_0.
\end{equation}
	
\emph{Step 2: here we show that $\lambda_0=0$.} To this end we assume that
$\lambda_0<0$ and we reach a contradiction by proving that $u_i\leq
u_{i,\lambda_0+\nu}$ in $\Omega_{\lambda_0+\nu}\setminus
R_{\lambda_0+\nu}(\Gamma)$ for any $0<\nu<\bar\nu$ for some small
$\bar\nu>0$ and for every $i=1,...,m$. By continuity we know that $u_i \leq u_{i, \lambda_0}$ in $\Omega_{\lambda_0}\setminus R_{\lambda_0}(\Gamma)$ for every $i=1,...,m$. Since $\Omega$ is convex in the $x_1-$direction and the set $ R_{\lambda_0}(\Gamma)$ lies in the hyperplane of equation $\{\, x_1 = - 2 \lambda_0 \, \}$, we see that $\Omega_{\lambda_0}\setminus R_{\lambda_0}(\Gamma)$ is open and connected. Moreover, using $(h_{f_i})-(ii)$ we have that
\[
\begin{split}
-\Delta(u_i-u_{i,\lambda_0}) &= f(u_1,...,u_m) - f(u_{1,\lambda_0},...,u_{m,\lambda_0}) \\
&= \left(f(u_1,...,u_m) - f(u_{1,\lambda_0},...,u_m)\right) + \cdots\\
&\cdots  + \left(f(u_{1,\lambda_0},...,u_m) - f(u_{1,\lambda_0},...,u_{m,\lambda_0}) \right) \leq 0.
\end{split}
\]
Therefore, by the strong maximum principle we deduce that
$u_i< u_{i,\lambda_0}$ in $\Omega_{\lambda_0}\setminus R_{\lambda_0}(\Gamma)$ and for every $i=1,...,m$. 
	
{Now, note that for $K\subset \Omega_{\lambda_0}\setminus R_{\lambda_0}(\Gamma)$, there is $\nu=\nu(K,\lambda_0)>0$, sufficiently small, such that $K \subset \Omega_{\lambda} \setminus R_{\lambda}(\Gamma)$ for every $ \lambda \in [\lambda_0, \lambda_0 + \nu].$ Consequently $u_i$ and $u_{i,\lambda}$ are well defined on $K$ for every $ \lambda \in [\lambda_0, \lambda_0 + \nu]$ and for every $i=1,...,m$. Hence, by the uniform continuity of the functions $ g_i(x,\lambda) := u_i(x) - u_i(2\lambda-x_1,x^{'}) $ on the compact set $K \times [\lambda_0, \lambda_0 + \nu]$ we can ensure that  $K \subset \Omega_{\lambda_0 + \nu} \setminus R_{\lambda_0 + \nu}(\Gamma)$ and
$u_i< u_{i,\lambda_0+\nu}$ in $K$ for any $0 \le \nu<\bar\nu$, for some $\bar\nu = \bar\nu(K,\lambda_0)>0$ small. Clearly we can also assume that $\bar\nu < \frac{\vert \lambda_0 \vert}{4}. $}

Let us consider $\psi_\varepsilon$ constructed in such a way that it vanishes in a neighborhood  of $R_{\lambda_0 + \nu}(\Gamma)$ and $\phi_{\tau}$ constructed in such a way it vanishes in a neighborhood  of $\gamma_{\lambda_0+\nu}=\partial \Omega \cap T_{\lambda_0+ \nu}$. {As shown in the proof of Lemma \ref{E:leaiuto2}, the functions} $$\varphi_i\,:= \begin{cases} \, \, w_{i,\lambda_0+\nu}^+\psi_\varepsilon^2 \phi_{\tau}^2 \, & \text{in}\quad {\Omega_{\lambda_0+\nu}} \\
0 &  \text{in}\quad {\R^n \setminus \Omega_{\lambda_0+\nu}} 
\end{cases}$$
{are such that $ \varphi_i \to  w_{i,\lambda_0+\nu}^+$ in $H^1_0(\Omega_{\lambda_0+\nu})$,  as $\varepsilon$ and $\tau$ tend to zero. Moreover, $\varphi_i \in C^{0,1}(\overline{\Omega_{\lambda_0+\nu}})$ and ${\varphi_i}_{\vert _{\partial \Omega_{\lambda_0+\nu}}} = 0$, by Lemma \ref{E:leaiuto}, and $ \varphi_i = 0 $ on an open neighborhood of $K$, by the above argument.
Therefore, $ \varphi_i \in H^1_0 (\Omega_{\lambda_0+\nu} \setminus K)$ and thus, also $w_{i,\lambda_0+\nu}^+$ belongs to  $H^1_0(\Omega_{\lambda_0+\nu} \setminus K)$. We also note that $ \nabla w_{i,\lambda_0+\nu}^+ = 0$ on an open neighborhood of $K$.}
	
\noindent {Now we argue as in Lemma \ref{E:leaiuto2} and we plug $\varphi_i$ as test function in
\eqref{E:weakfor} and \eqref{E:weakforRef} so that, by subtracting, we get  }
\begin{equation}\label{E:subtraction3} \nonumber
\begin{split}
\int_{\Omega_{\lambda_0+\nu}}|\nabla w_{i,\lambda_0+\nu}^+ &|^2
\psi_\varepsilon^2 \phi_{\tau}^2 \,dx \\
&=  - 2\int_{\Omega_{\lambda_0+\nu}}\nabla w_{i,\lambda_0+\nu}^+ \nabla \psi_\varepsilon w_{i,\lambda_0+\nu}^+ \psi_\varepsilon \phi_{\tau}^2 \,dx\\
& - 2\int_{\Omega_{\lambda_0+\nu}}\nabla w_{i,\lambda_0+\nu}^+ \nabla \phi_\tau w_{i,\lambda_0+\nu}^+ \psi_\varepsilon^2 \phi_{\tau}\,dx\\
& +\int_{\Omega_{\lambda_0+\nu}} [f_i(u_1,u_2,...,u_m)-f_i(u_{1,\lambda_0+\nu},u_{2,\lambda_0+\nu},...,u_{m,\lambda_0+\nu})] w_{i,\lambda_0+\nu}^+\psi_\varepsilon^2  \phi_{\tau}^2\, dx.\\
\end{split}
\end{equation}
Therefore, taking into account the properties of $ w_{\lambda_0+\nu}^+$ and $ \nabla w_{\lambda_0+\nu}^+$ we also have 	

\begin{equation}\nonumber
\begin{split} 
\int_{\Omega_{\lambda_0+\nu}\setminus K}|\nabla &w_{i,\lambda_0+\nu}^+ |^2
\psi_\varepsilon^2 \phi_{\tau}^2 \,dx \\
&\leq-2\int_{\Omega_{\lambda_0+\nu}\setminus K}\nabla w_{i,\lambda_0+\nu}^+\nabla \psi_\varepsilon w_{i,\lambda_0+\nu}^+ \psi_\varepsilon \phi_{\tau}^2 \,dx \\
&- 2\int_{\Omega_{\lambda_0+\nu}\setminus K}\nabla
w_{i,\lambda_0+\nu}^+\nabla \phi_\tau w_{i,\lambda_0+\nu}^+
\psi_\varepsilon^2 \phi_{\tau} \,dx\\
&+ \int_{\Omega_{\lambda_0+\nu}\setminus K} [f_i(u_1,u_2,...,u_m)-f_i(u_{1,\lambda_0+\nu},u_{2,\lambda_0+\nu},...,u_{m,\lambda_0+\nu})] w_{i,\lambda_0+\nu}^+\psi_\varepsilon^2  \phi_{\tau}^2\, dx.\\
\end{split}
\end{equation}
Furthermore, since $f_i$ are $\mathcal{C}^1$ functions, we deduce that
\begin{equation}\label{E:jkdfh} \nonumber
\begin{split}
\int_{\Omega_{\lambda_0+\nu}\setminus K}|\nabla
w_{i,\lambda_0+\nu}^+|^2\psi_\varepsilon^2 \phi_{\tau}^2\,dx&\leq	2\int_{\Omega_{\lambda_0+\nu}\setminus K}|\nabla w_{i,\lambda_0+\nu}^+| |\nabla \psi_\varepsilon| w_{i,\lambda_0+\nu}^+ \psi_\varepsilon \phi_{\tau}^2 \,dx \\
&+ 2\int_{\Omega_{\lambda_0+\nu}\setminus K} |\nabla
w_{i,\lambda_0+\nu}^+| |\nabla \phi_\tau| w_{i,\lambda_0+\nu}^+	\psi_\varepsilon^2 \phi_{\tau}
\,dx\\
&+ \sum_{j=1}^m C_j(f_i) \int_{\Omega_{\lambda_0+\nu}\setminus K}  w_{j,\lambda_0+\nu}^+ w_{i,\lambda_0+\nu}^+ \psi_\varepsilon^2 \phi_{\tau}^2\,dx.
\end{split}
\end{equation}
	
Now, as in the proof of Lemma \ref{E:leaiuto2}, we use Young's
inequality to deduce that
\begin{equation} \nonumber
\begin{split}
\int_{\Omega_{\lambda_0+\nu}\setminus K}
\vert \nabla w_{i,\lambda_0+\nu}^+|^2\psi_\varepsilon^2
\phi_{\tau}^2 \,dx &\leq 8 \int_{\Omega_{\lambda_0+\nu}\setminus K}|\nabla \psi_\varepsilon|^2 (w_{i,\lambda_0+\nu}^+)^2 \phi_{\tau}^2 \,dx \\
&+ 8 \int_{\Omega_{\lambda_0+\nu}\setminus K}|\nabla \phi_\tau|^2
(w_{i,\lambda_0+\nu}^+)^2 \psi_\varepsilon^2\, dx\\
&+ \sum_{j=1}^m C_j \int_{\Omega_{\lambda_0+\nu}\setminus K}  w_{j,\lambda_0+\nu}^+ w_{i,\lambda_0+\nu}^+ \psi_\varepsilon^2 \phi_{\tau}^2\,dx,
\end{split}
\end{equation}
which in turns yields
\begin{equation} \nonumber
\begin{split}
\sum_{i=1}^m \int_{\Omega_{\lambda_0+\nu}\setminus K}
\vert \nabla w_{i,\lambda_0+\nu}^+|^2\psi_\varepsilon^2
\phi_{\tau}^2 \,dx &\leq 32 (\epsilon +\tau) \sum_{i=1}^m \|u_i\|^2_{L^\infty(\Omega_{\lambda_0+ \bar\nu})}\\
&+\frac{m}{2} \sum_{i=1}^m (1+C_i^2)
\int_{\Omega_{\lambda_0+\nu}\setminus K}  (w_{i,\lambda_0+\nu}^+)^2 \psi_\varepsilon^2 \phi_{\tau}^2\,dx.\\
\end{split}	
\end{equation}
Passing to the limit, as $(\epsilon, \tau) \to (0,0),$ in the latter we get
\begin{equation}\label{E:sdjfhsfskl}\nonumber
\begin{split}
&\sum_{i=1}^m\int_{\Omega_{\lambda_0+\nu}\setminus K}|\nabla
w_{i,\lambda_0+\nu}^+|^2\,dx \leq \frac{m}{2} \sum_{i=1}^m (1+C_i^2)
\int_{\Omega_{\lambda_0+\nu}\setminus K}  (w_{i,\lambda_0+\nu}^+)^2 \,dx.\\
&\leq \frac{m}{2} \sum_{i=1}^m (1+C_i^2)
C_{i,p}^2(\Omega_{\lambda_0+\nu}\setminus K)
\int_{\Omega_{\lambda_0+\nu}\setminus K}   |\nabla
w_{i,\lambda_0+\nu}^+|^2\,dx\, ,
\end{split}
\end{equation}
where $C_ {i,p}(\cdot)$ are the Poincar\'e constants (in the Poincar\'e inequalities in $H^1_0(\Omega_{\lambda_0+\nu}\setminus K)$).
Now we recall that $C_{i,p}^2(\Omega_{\lambda_0+\nu}\setminus K) \le Q(n) \vert \Omega_{\lambda_0+\nu}\setminus K \vert ^{\frac{2}{n}}$ for every $i=1,...,m$, where $Q=Q(n)$ is a positive constant depending only on the dimension $n$, and therefore, by summarizing, we have proved that for every compact set $K\subset \Omega_{\lambda_0}\setminus R_{\lambda_0}(\Gamma)$ there is a small $\bar\nu = \bar\nu(K, \lambda_0) \in (0, \frac{\vert \lambda_0 \vert}{4})$ such that for every $ 0 \le \nu < \bar\nu$ we have
\begin{equation}\label{E:sdjfhsfskl2} 
\begin{split}
\sum_{i=1}^m \int_{\Omega_{\lambda_0+\nu}\setminus K} & |\nabla
w_{i, \lambda_0+\nu}^+|^2\,dx \leq \frac{m}{2} \sum_{i=1}^m (1+C_i^2) Q(n) \vert \Omega_{\lambda_0+\nu}\setminus K \vert ^{\frac{2}{n}}
\int_{\Omega_{\lambda_0+\nu}\setminus K}   |\nabla
w_{\lambda_0+\nu}^+|^2\,dx.
\end{split}
\end{equation}
Now we first fix a compact $K\subset \Omega_{\lambda_0}\setminus R_{\lambda_0}(\Gamma)$ such that
$$\vert \Omega_{\lambda_0}\setminus K \vert^{\frac{2}{n}} < [m (1+C_i^2) Q(n)]^{-1} \; \text{for every $i=1,...,m$,}$$ this is possible since $ \vert  R_{\lambda_0}(\Gamma)\vert  =0$ by the assumption on $ \Gamma$, and then we take $ \bar\nu_0 < \bar\nu $ such that for every $ 0 \le \nu < \bar\nu_0$ we have
$\vert \Omega_{\lambda_0 + \nu} \setminus \Omega_{\lambda_0}
\vert^{\frac{2}{n}}< [ 4m (1+C_i^2) Q(n)]^{-1}$. Inserting those
informations into \eqref{E:sdjfhsfskl2} we immediately get that
\begin{equation}\label{E:sdjfhsfskl3} \nonumber
\begin{split}
&\int_{\Omega_{\lambda_0+\nu}\setminus K}|\nabla
w_{i,\lambda_0+\nu}^+|^2\,dx  < \frac{1}{2}
\int_{\Omega_{\lambda_0+\nu}\setminus K}   |\nabla
w_{i, \lambda_0+\nu}^+|^2\,dx
\end{split}
\quad \text{for every $i=1,...,m$}
\end{equation}
and so $\nabla w_{i,\lambda_0+\nu}^+=0$ on $ \Omega_{\lambda_0+\nu} \setminus K$ for every $ 0 \le \nu < \bar\nu_0$ and $i=1,...,m$. On the other hand, we recall that $\nabla w_{i,\lambda_0+\nu}^+=0$ on an open neighborhood of $K$ for every $0 \le \nu < \bar\nu$ and $i=1,...,m$, thus $\nabla w_{i,\lambda_0+\nu}^+=0$ on $ \Omega_{\lambda_0+\nu} $ for every $ 0 \le \nu < \bar\nu_0$ and $i=1,...,m$. The latter proves that $u_i\leq u_{i,\lambda_0+\nu}$ in
$\Omega_{\lambda_0+\nu}\setminus R_{\lambda_0+\nu}(\Gamma)$ for every
$0<\nu<\bar\nu_0$ and $i=1,...,m$. Such a contradiction shows that
\[
\lambda_0=0\,.
\]

\emph{Step 3: conclusion.} Since the moving plane procedure can be
performed in the same way but in the opposite direction, then this
proves the desired symmetry result.  The fact that the solution is
increasing in the $x_1$-direction in $\{x_1<0\}$ is implicit in the
moving plane procedure. Since $u$ has $C^1$ regularity, the fact that $\partial_{x_1}u_i$ is positive for $x_1<0$
follows by the maximum principle, the H\"opf
lemma and the assumption $(h_{f_i})$.
	
\end{proof}

\section{Proof of Theorem \ref{E:main2}}\label{E:mainproof2}

\begin{proof}[Proof of Theorem \ref{E:main2}]
We first note that, thanks to a well-known result of Brezis and Kato \cite{BK} and standard elliptic estimates (see also \cite{Stru}),	the solution $(u_1,...,u_m)$ to \eqref{E:criticalsingularequbdd} is smooth in $\R^n \setminus \Gamma$.	Furthermore we observe that it is enough to
prove the theorem for the special case in which	{\textit {the origin does not belong to}} $\Gamma$. Indeed, if the result is true in this special case, then we can apply it to the functions $u^{(i)}_z(x) := u_i(x+z)$ for every $i=1,...,m$, where $ z \in \{x_1=0\} \setminus	\Gamma \neq \emptyset$, which satisfies the system
\eqref{E:criticalsingularequbdd} with $\Gamma$ replaced by $-z + \Gamma$ (note
that $-z + \Gamma$ is a closed and proper subset of $\{x_1=0\}$ with
$\underset{\R^n}{\operatorname{Cap}_2}(-z +\Gamma)=0$ and such that
the origin does not belong to it).

\noindent Under this assumption, we consider the map $K: \R^n \setminus \{0\} \longrightarrow \R^n  \setminus \{0\} $ defined by $K(x):=\frac{x}{|x|^2}$. Given  $(u_1,...,u_m)$ solution to \eqref{E:criticalsingularequbdd}, the Kelvin transform of $u_i$ is given by
\begin{equation}\label{E:kelvcritical1} \hat{u}_i(x):= \frac{1}{|x|^{n-2}} u_i\left(\frac{x}{|x|^2}\right) \quad x \in \R^n \setminus \{\Gamma^* \cup \{0\}\},
\end{equation}
where $\Gamma^* = K( \Gamma)$ and $i=1,...,m$. It follows that $(\hat{u}_1,...,\hat{u}_m)$ weakly satisfies \eqref{E:criticalsingularequbdd} in $\R^n
\setminus \{\Gamma^* \cup \{0\}\}$ (i.e. in the sense that it satisfies \eqref{E:weakforubdd}) and that $\Gamma^*\subset
\{x_1=0\}$ since, by assumption, $\Gamma\subset \{x_1=0\}$.
Furthermore, we also have that $\Gamma^*$ is bounded (not
necessarily closed) since we assumed that $0\notin\Gamma$.

To proceed further we recall some useful lemma whose proofs are contained in \cite{EFS}.

\begin{lem}[\cite{EFS}]\label{E:diffeo} Let $F: \R^n \setminus \{0\} \longrightarrow \R^n \setminus \{0\}$ be a $C^1-$diffeomorphism and let $A$ be a bounded open set of $\R^n \setminus \{0\}$. If $C \subset A$ is a compact set such that
	\begin{equation}\label{E:closedsetnullcap} \nonumber
	\underset{A}{\operatorname{Cap}_2}(C)=0,
	\end{equation}
	then
	\begin{equation}\label{E:kelvclosedsetnullcap} \nonumber
	\underset{F(A)}{\operatorname{Cap}_2}(F(C))=0.
	\end{equation}
\end{lem}

\begin{lem}[\cite{EFS}]\label{E:geom1} Let $\Gamma $ be a closed subset of $\R^n$, with $ n \ge 3$. Also suppose that $0 \not\in \Gamma$ and
	\begin{equation}\label{E:gamm}\nonumber
	\underset{\R^n}{\operatorname{Cap}_2}(\Gamma)=0.
	\end{equation}
	Then
	\begin{equation}\label{E:gamstar} \nonumber
	\underset{\R^n}{\operatorname{Cap}_2}(\Gamma^*)=0.
	\end{equation}
\end{lem}

\noindent  Let us now fix some notations. We set
\begin{equation}\label{E:eq:sn2hgjhsdgf1} \nonumber
\Sigma_\lambda=\{x\in\mathbb{R}^n\,:\,x_1 <\lambda\}\,.
\end{equation}
As above $x_\lambda=(2\lambda-x_1,x_2,\ldots,x_n)$ is the reflection
of $x$ through the hyperplane $T_\lambda=\{x=(x_1,...,x_n)\in \R^n \
| \ x_1=\lambda\}$. Finally we consider the Kelvin transform $(\hat{u}_1, ..., \hat{u}_m)$ of $(u_1,...,u_m)$ defined in \eqref{E:kelvcritical1} and we set
\begin{equation}\label{E:mov1} \nonumber
w_{i,\lambda}^+=(\hat{u}_i- \hat{u}_{i,\lambda})^+
\end{equation}
where $i=1,...,m$. Note that $(\hat{u}_1,...,\hat{u}_m)$ weakly solves
\begin{equation}\label{E:crittest}
\int_{\R^n} \nabla \hat{u}_i \nabla \vp_i \, dx = \sum_{j=1}^m a_{ij} \int_{\R^n} \hat{u}_j^{2^*-1} \vp_i \, dx \qquad \forall \varphi_i\in C^{1}_c(\mathbb{R}^n\setminus
\Gamma^* \cup \{0\})\,.
\end{equation}

and $(\hat{u}_{1,\lambda},...,\hat{u}_{m,\lambda})$ weakly solves
\begin{equation}\label{E:reflcriticaltest}
\int_{\R^n} \nabla \hat{u}_{i,\lambda} \nabla \vp_i \, dx = \sum_{j=1}^m a_{ij} \int_{\R^n} \hat{u}_{j,\lambda}^{2^*-1} \vp_i \, dx \qquad \forall \varphi_i \in C^{1}_c(\mathbb{R}^n\setminus
R_\lambda(\Gamma^* \cup \{0\}))\,.
\end{equation}

where $i=1,...,m$. The properties of the Kelvin transform, the fact that
$0\notin\Gamma$ and the regularity of $u_i$ imply that $\vert \hat{u}_i(x)
\vert \le C \vert x \vert^{2-N} $ for every $x \in \R^n$ and $i=1,...,m$ such that
$\vert x \vert \ge R$, where $C$ and $R$ are positive constants
(depending on $u_i$).  In particular, for every $ \lambda <0$, we have
\begin{equation}\label{E:kelvinpn} \nonumber
\hat{u}_i\in L^{2^*}(\Sigma_\lambda)\cap L^{\infty}(\Sigma_\lambda) \cap C^0 (\overline{\Sigma_\lambda})
\end{equation}
for every $i=1,...,m$. We will prove the result by showing that, actually, it holds $\hat{w}_{i,\lambda}^+ \equiv 0$ for every $i=1,...,m$. To prove this, we have to perform the moving plane method. 

\begin{lem}\label{E:stimgradcrit} Under the assumption of Theorem
\ref{E:main2}, for every $\lambda < 0$, we have that $\hat{w}_{i,\lambda}^+ \in
L^{2^*}(\Sigma_\lambda), \nabla \hat{w}_{i,\lambda}^+  \in L^2(\Sigma_\lambda)$ and
\begin{equation}\label{E:buonastimacrit}
\sum_{i=1}^m \|w_{i,\lambda}^+\|^2_{L^{2^*}(\Sigma_\lambda)} \leq \sum_{i=1}^m C_{i,S}^2\int_{\Sigma_\lambda}|\nabla \hat{w}_{i,\lambda}^+|^2\,dx \leq	 2\frac{n+2}{n-2}  \sum_{i,j=1}^m a_{ij} C_{i,S}^2 \|\hat{u}_j\|^{2^*-1}_{L^{2^*}(\Sigma_\lambda)} \|\hat{u}_i\|_{L^{2^*}(\Sigma_\lambda)}.
\end{equation}
where $C_{i,S}$ are the best constants in Sobolev embeddings.
\end{lem}

	\begin{proof}
	We immediately see that $w_{i,\lambda}^+\in L^{2^*}(\Sigma_\lambda),$
	since $0\leq w_{i,\lambda}^+\leq \hat{u}_i \in L^{2^*}(\Sigma_\lambda)$ for every $i=1,...,m$. The rest of the proof follows the lines of the one of Lemma
	\ref{E:leaiuto2}. Arguing as in section 2, for every $ \varepsilon>0$,
	we can find a function $\psi_\varepsilon \in C^{0,1}(\R^N, [0,1])$
	such that
	$$\int_{\Sigma_\lambda} |\nabla \psi_\varepsilon|^2 < 4 \varepsilon$$
	and $\psi_\varepsilon = 0$ in an open neighborhood
	$\mathcal{B_{\varepsilon}}$ of $R_{\lambda}(\{\Gamma^* \cup
	\{0\}\})$, with $\mathcal{B_{\varepsilon}} \subset
	\Sigma_{\lambda}$.
	
	Fix $ R_0>0$ such that $R_{\lambda}(\{\Gamma^* \cup \{0\} )\subset
	B_{R_0} $ and, for every $ R > R_0$, let $\eta_R$ be a standard
	cut off function such that $ 0 \le \eta_R \le 1 $ on $ \R^n$,
	$\eta_R=1$ in $B_R$, $\eta_R=0$ outside $B_{2R}$ with
	$|\nabla\eta_R|\leq 2/R,$ and consider
	$$\varphi_i\,:= \begin{cases}
	\, w_{i,\lambda}^+\psi_\varepsilon^2\eta_R^2 \, & \text{in}\quad\Sigma_\lambda, \\
	0 &  \text{in}\quad \R^n \setminus  \Sigma_\lambda
	\end{cases}$$
	for every $i=1,...,m$. Now, as in Lemma \ref{E:leaiuto} we see that $ \varphi_i \in C^{0,1}_c(\R^n)$ with $supp(\varphi_i)$ contained in $\overline {\Sigma_{\lambda} \cap B_{2R}} \setminus R_{\lambda}(\{\Gamma^* \cup \{0\}\})$ and
	\begin{equation}\label{E:gradvarphii-intero}
	\nabla \varphi_i = \psi_\varepsilon^2 \eta_R^2 \nabla w_{i,\lambda}^+ +
	2 w_{i,\lambda}^+  (\psi_\varepsilon^2 \eta_R \nabla \eta_R + \psi_\varepsilon \eta_R^2  \nabla  \psi_\varepsilon).
	\end{equation}
	Therefore, by a standard density argument, we can use $\varphi_i$ as test functions respectively in \eqref{E:crittest} and in \eqref{E:reflcriticaltest} so that, subtracting we get
	\begin{equation}\label{E:subt1}
	\begin{split}
	\int_{\Sigma_\lambda}|\nabla w_{i,\lambda}^+|^2 \psi_\varepsilon^2\eta_R^2\,dx&=
	-2\int_{\Sigma_\lambda}\nabla w_{i,\lambda}^+ \nabla \psi_\varepsilon
	 w_{i,\lambda}^+ \psi_\varepsilon\eta_R^2\,dx
	-2\int_{\Sigma_\lambda} \nabla w_{i,\lambda}^+ \nabla \eta_R w_{i,\lambda}^+ \eta_R \psi_\varepsilon^2\,dx\\
	&+\sum_{i=1}^m a_{ij}\int_{\Sigma_\lambda} (\hat{u}_j^{2^*-1}-\hat{u}_{j,\lambda}^{2^*-1}) w_{i,\lambda}^+ \psi_\varepsilon^2\eta_R^2\,dx\,\,\\
	&=:\,I_1+I_2+I_3\,.
	\end{split}
	\end{equation}

	Exploiting also Young's inequality and recalling that
	$0\leq  w_{i,\lambda}^+\leq \hat{u}_i$, we get that
	\begin{equation}\label{E:Ii1}
	\begin{split}
	|I_1|&\leq \frac{1}{4} \int_{\Sigma_\lambda}|\nabla w_{i,\lambda}^+|^2 \psi_\varepsilon^2\eta_R^2\,dx
	+4\int_{\Sigma_\lambda}| \nabla \psi_\varepsilon|^2(w_{i,\lambda}^+)^2\eta_R^2\,dx\\
	&\leq  \frac{1}{4} \int_{\Sigma_\lambda}|\nabla  w_{i,\lambda}^+|^2 \psi_\varepsilon^2 \eta_R^2\,dx + 16 \varepsilon \|\hat{u}_i\|^2_{L^\infty(\Sigma_\lambda)} .\\
	\end{split}
	\end{equation}

	Furthermore we have that
	\begin{equation}\label{E:Ii2}
	\begin{split}
	|I_2|&\leq \frac{1}{4} \int_{\Sigma_\lambda}|\nabla w_{i,\lambda}^+|^2 \psi_\varepsilon^2\eta_R^2\,dx
	+4\int_{\Sigma_\lambda\cap(B_{2R}\setminus B_{R})}|\nabla \eta_R|^2 (w_{i,\lambda}^+)^2 \psi_\varepsilon^2\,dx\\
	&\leq  \frac{1}{4} \int_{\Sigma_\lambda}|\nabla w_{i,\lambda}^+|^2 \psi_\varepsilon^2 \eta_R^2\,dx \\
	&+ 4\left(\int_{\Sigma_\lambda\cap(B_{2R}\setminus B_{R})}|\nabla
	\eta_R|^n\,dx\right)^{\frac{2}{n}}
	\left(\int_{\Sigma_\lambda\cap(B_{2R}\setminus B_{R})}\hat{u}_i^{2^*}\,dx\right)^{\frac{n-2}{n}}\\
	&\leq  \frac{1}{4} \int_{\Sigma_\lambda}|\nabla w_{i,\lambda}^+|^2 \psi_\varepsilon^2\eta_R^2\,dx\,+\,
	c(n) \left(\int_{\Sigma_\lambda\cap(B_{2R}\setminus B_{R})} \hat{u}_i^{2^*}\,dx\right)^{\frac{n-2}{n}}
	\end{split}
	\end{equation}
	where $c(n)$ is a positive constant depending only on the dimension $n$. Let us now estimate $I_3$. Since $\hat{u}_i(x), \hat{u}_{i, \lambda}(x)>0$, by the
	convexity of $ t \to t^{2^*-1},$ for $ t >0$, we obtain
	$$\hat{u}_i^{2^*-1}(x)-\hat{u}_{i,\lambda}^{2^*-1}(x) \le  \frac{n+2}{n-2}
	\hat{u}_{i,\lambda}^{2^*-2}(x) (\hat{u}_i(x) - \hat{u}_{i,\lambda}(x)),$$ for every $x \in
	\Sigma_{\lambda}$ and $i=1,...,m$. Thus, by making use of the monotonicity of $ t
	\to t^{2^*-2}$, for $ t >0$ and the definition of $ w_{i,\lambda}^+$ we
	get  
	$$(\hat{u}_i^{2^*-1}-\hat{u}_{i,\lambda}^{2^*-1}) w_{i,\lambda}^+ \le  \frac{n+2}{n-2} \hat{u}_{i,\lambda}^{2^*-2}(\hat{u}_i-\hat{u}_{i,\lambda})  w_{i,\lambda}^+ \le \frac{n+2}{n-2}	\hat{u}_i^{2^*-2}(w_{i,\lambda}^+)^2$$
	for every $i=1,...,m$. Therefore
	\begin{equation}\label{E:Ii3}
	\begin{split}
	|I_3|& \leq
	\frac{n+2}{n-2} \sum_{j=1}^m a_{ij}\int_{\Sigma_\lambda}
	\hat{u}_j^{2^*-2}w_{j,\lambda}^+w_{i,\lambda}^+ \psi_\varepsilon^2\eta_R^2\,dx\,\\
	&\leq \frac{n+2}{n-2}  \sum_{j=1}^m a_{ij} \int_{\Sigma_\lambda} \hat{u}_j^{2^*-2}\hat{u}_j \hat{u}_i dx\,=
	\frac{n+2}{n-2}  \sum_{j=1}^m a_{ij} \int_{\Sigma_\lambda} \hat{u}_j^{2^*-1}u_i\,dx\\
	& = \frac{n+2}{n-2} \left( a_{ii}\|\hat{u}_i\|^{2^*}_{L^{2^*}(\Sigma_\lambda)} + \sum_{\substack{ j=1 \\ j \neq i}}^m a_{ij} \int_{\Sigma_\lambda} \hat{u}_j^{2^*-1}u_i\,dx   \right)\\
	& \leq \frac{n+2}{n-2} \left( a_{ii}\|\hat{u}_i\|^{2^*}_{L^{2^*}(\Sigma_\lambda)} + \sum_{\substack{ j=1 \\ j \neq i}}^m a_{ij} \left(\int_{\Sigma_\lambda} \hat{u}_j^{2^*}\,dx\right)^{\frac{n+2}{2n}} \left( \int_{\Sigma_\lambda} u_i^{2^*}\,dx \right)^{\frac{1}{2^*}}    \right)\\
	&= \frac{n+2}{n-2}  \sum_{j=1}^m a_{ij} \|\hat{u}_j\|^{2^*-1}_{L^{2^*}(\Sigma_\lambda)} \|\hat{u}_i\|_{L^{2^*}(\Sigma_\lambda)} 
	\end{split}
	\end{equation}
	where we also used that $0\leq  w_{i,\lambda}^+ \leq \hat{u}_i$ for every $i=1,...,m$ and H\"older inequality.

	Taking into account the estimates on $I_1$, $I_2$ and $I_3$, by \eqref{E:subt1} we deduce that
	
	\begin{equation}\label{E:subt2} \nonumber
	\begin{split}
	\int_{\Sigma_\lambda}|\nabla w_{i,\lambda}^+|^2 \psi_\varepsilon^2\eta_R^2\,dx& 
	\leq  32 \varepsilon \|\hat{u}_i\|^2_{L^\infty(\Sigma_\lambda)} +  2c(n) \left(\int_{\Sigma_\lambda\cap(B_{2R}\setminus B_{R})} \hat{u}_i^{2^*}\,dx\right)^{\frac{n-2}{n}}\\
	&+ 2\frac{n+2}{n-2}  \sum_{j=1}^m a_{ij} \|\hat{u}_j\|^{2^*-1}_{L^{2^*}(\Sigma_\lambda)} \|\hat{u}_i\|_{L^{2^*}(\Sigma_\lambda)}\\
	\end{split}
	\end{equation}
	which in turns yields
	\begin{equation}\label{E:subt3}
	\begin{split}
	\sum_{i=1}^m \int_{\Sigma_\lambda}|\nabla w_{i,\lambda}^+|^2 \psi_\varepsilon^2\eta_R^2\,dx& 
	\leq  32 \varepsilon \sum_{i=1}^m \|\hat{u}_i\|^2_{L^\infty(\Sigma_\lambda)} +  2c(n) \sum_{i=1}^m \| \hat{u}_i \|^2_{\Sigma_\lambda\cap(B_{2R}\setminus B_{R})}\\
	&+ 2\frac{n+2}{n-2}  \sum_{i=1}^m\sum_{j=1}^m a_{ij} \|\hat{u}_j\|^{2^*-1}_{L^{2^*}(\Sigma_\lambda)} \|\hat{u}_i\|_{L^{2^*}(\Sigma_\lambda)}.\\
	\end{split}
	\end{equation}

	By Fatou Lemma, as $\varepsilon$ tends to zero and $R$ tends to
	infinity,  we deduce that $\nabla w_{i,\lambda}^+ \in L^2(\Sigma_{\lambda})$ for every $i=1,...,m$.  We also note that $ \varphi_i \to
	w_{i,\lambda}^+$ in $L^{2^*}(\Sigma_\lambda)$, by definition of
	$\varphi_i$, and that $\nabla \varphi_i \to \nabla w_{i,\lambda}$ in $L^2(\Sigma_{\lambda})$, by
	\eqref{E:gradvarphii-intero} and the fact that $w_{i,\lambda}^+\in
	L^{2^*}(\Sigma_\lambda)$ for every $i=1,...,m$. Therefore by \eqref{E:subt3} we have
	\begin{equation}\label{E:final}
	\sum_{i=1}^m \int_{\Sigma_\lambda}|\nabla w_{i,\lambda}^+|^2 \,dx \leq 2\frac{n+2}{n-2}  \sum_{i,j=1}^m a_{ij} \|\hat{u}_j\|^{2^*-1}_{L^{2^*}(\Sigma_\lambda)} \|\hat{u}_i\|_{L^{2^*}(\Sigma_\lambda)}.
	\end{equation}
	Now we apply the Sobolev embedding theorem to \eqref{E:final}, to deduce 
	\eqref{E:buonastimacrit}.
	
\end{proof}

We can now complete the proof of Theorem \ref{E:main2}. As for the proof of Theorem \ref{E:semilinearSym},  we split the proof into three steps and we start with

\noindent \emph{Step 1: there exists $M>1$ such that $\hat{u}_i \leq \hat{u}_{i,\lambda}$ in $\Sigma_\lambda\setminus R_\lambda(\Gamma^* \cup \{0\})$, for all $\lambda< -M$ and $i=1,...,m$.}

Arguing as in the proof of Lemma \ref{E:stimgradcrit} and using the same notations and the same construction for $\psi_\varepsilon$, $\eta_R$ and $\varphi_i$, we get
\begin{equation}\label{E:start}  \nonumber
\begin{split}
\int_{\Sigma_\lambda}|\nabla w_{i,\lambda}^+|^2 \psi_\varepsilon^2\eta_R^2\,dx&=
-2\int_{\Sigma_\lambda}\nabla w_{i,\lambda}^+ \nabla \psi_\varepsilon
w_{i,\lambda}^+ \psi_\varepsilon\eta_R^2\,dx
-2\int_{\Sigma_\lambda} \nabla w_{i,\lambda}^+ \nabla \eta_R w_{i,\lambda}^+ \eta_R \psi_\varepsilon^2\,dx\\
&+\sum_{i=1}^m a_{ij}\int_{\Sigma_\lambda} (\hat{u}_j^{2^*-1}-\hat{u}_{j,\lambda}^{2^*-1}) w_{i,\lambda}^+ \psi_\varepsilon^2\eta_R^2\,dx\,\,\\
&=:\,I_1+I_2+I_3\,.
\end{split}
\end{equation}
where $I_1,I_2$ and $I_3$ can be estimated exactly as in \eqref{E:Ii1},
\eqref{E:Ii2} and \eqref{E:Ii3}. The latter yield
\begin{equation} \nonumber
\begin{split}
\sum_{i=1}^m \int_{\Sigma_\lambda}|\nabla w_{i,\lambda}^+|^2 \psi_\varepsilon^2\eta_R^2\,dx& 
\leq  32 \varepsilon \sum_{i=1}^m \|\hat{u}_i\|^2_{L^\infty(\Sigma_\lambda)} +  2c(n) \sum_{i=1}^m \| \hat{u}_i \|^2_{\Sigma_\lambda\cap(B_{2R}\setminus B_{R})}\\
&+ 2\frac{n+2}{n-2} \sum_{i,j=1}^m a_{ij}\int_{\Sigma_\lambda}
\hat{u}_j^{2^*-2}w_{j,\lambda}^+w_{i,\lambda}^+ \psi_\varepsilon^2\eta_R^2\,dx.
\end{split}
\end{equation}
Taking the limit in the latter, as $\varepsilon$ tends to zero and $R$ tends to infinity, leads to
\begin{equation} \nonumber
\begin{split}
\sum_{i=1}^m \int_{\Sigma_\lambda}|\nabla w_{i,\lambda}^+|^2 \,dx&\leq 2\frac{n+2}{n-2} \sum_{i,j=1}^m a_{ij}\int_{\Sigma_\lambda}
\hat{u}_j^{2^*-2}w_{j,\lambda}^+w_{i,\lambda}^+\,dx  < +\infty
\end{split}
\end{equation}
which combined with Lemma \ref{E:stimgradcrit} gives
\begin{equation}\label{E:aieiei}
\begin{split}
&\sum_{i=1}^m \int_{\Sigma_\lambda}|\nabla w_{i,\lambda}^+|^2 \leq 2\frac{n+2}{n-2} \sum_{i,j=1}^m a_{ij}\int_{\Sigma_\lambda}
\hat{u}_j^{2^*-2}w_{j,\lambda}^+w_{i,\lambda}^+\,dx\,\\
&\leq \frac{n+2}{n-2} \sum_{i,j=1}^m a_{ij} \left(\int_{\Sigma_\lambda}
\hat{u}_j^{2^*-2}(w_{j,\lambda}^+)^2\,dx\, + \int_{\Sigma_\lambda}
\hat{u}_j^{2^*-2}(w_{i,\lambda}^+)^2\,dx\, \right)\\
&\leq  \frac{n+2}{n-2} \sum_{i,j=1}^m a_{ij}\left[\left(\int_{\Sigma_\lambda} \hat{u}_j^{2^*}
\,dx\right)^{\frac{2}{n}} \left(\int_{\Sigma_\lambda}
(w_{j,\lambda}^+)^{2^*}\,dx\right)^{\frac{2}{2^*}} + \left(\int_{\Sigma_\lambda} \hat{u}_j^{2^*}
\,dx\right)^{\frac{2}{n}} \left(\int_{\Sigma_\lambda}
(w_{i,\lambda}^+)^{2^*}\,dx\right)^{\frac{2}{2^*}} \right]\\
&\leq  \frac{n+2}{n-2} \sum_{i=1}^m  \sum_{j=1}^m a_{ij} \| \hat{u}_j\|_{L^{2^*}(\Sigma_\lambda)}^{2^*-2} \left( C_{j,S}^2 \int_{\Sigma_\lambda} |\nabla w_{j,\lambda}^+|^2\,dx + C_{i,S}^2 \int_{\Sigma_\lambda} |\nabla w_{i,\lambda}^+|^2\,dx\right)\\
&= \sum_{i=1}^m \left[\frac{n+2}{n-2} \sum_{j=1}^m a_{ij} \left(2\delta_{ij} C_{i,S}^2 \|\hat{u}_i\|_{L^{2^*}(\Sigma_\lambda)}^{2^*-2} + (1-\delta_{ij}) C_{j,S}^2\| \hat{u}_j \|_{L^{2^*}(\Sigma_\lambda)}^{2^*-2}\right)\right] \int_{\Sigma_\lambda} |\nabla w_{i,\lambda}^+|^2\,dx,
\end{split}
\end{equation}
where 
$$\delta_{ij}:=\begin{cases}
1 & \quad \text{if} \; i = j \\
0 & \quad \text{if} \; i \neq j.
\end{cases}$$
Recalling that $\hat{u}_i, \hat{u}_j\in L^{2^*}(\Sigma_\lambda)$ for every $i,j=1,...,m$, we deduce the existence of $ M>1$ such that
\[
\frac{n+2}{n-2} \sum_{j=1}^m a_{ij} \left(2\delta_{ij} C_{i,S}^2 \|\hat{u}_i\|_{L^{2^*}(\Sigma_\lambda)}^{2^*-2} + (1-\delta_{ij}) C_{j,S}^2\| \hat{u}_j \|_{L^{2^*}(\Sigma_\lambda)}^{2^*-2}\right)<1\,
\]
for every $\lambda < -M$ and $i=1,....,m$. The latter and \eqref{E:kushfkusfksf} lead
to
\[
\int_{\Sigma_\lambda}|\nabla w_{i,\lambda}^+|^2\,dx=0\,.
\]
This implies that for every $i=1,...,m$ we have $w_{i,\lambda}^+=0$ by Lemma \ref{E:stimgradcrit} and the
claim is proved.\\

\noindent To proceed further we
define
\begin{equation}\nonumber
\Lambda_0:=\{\lambda<0 : \hat{u}_i \leq
\hat{u}_{i,t}\,\,\,\text{in}\,\,\,\Sigma_t\setminus R_t(\Gamma^* \cup
\{0\})\,\,\,\text{for all $t\in(a,\lambda]$ and for every i=1,...,m.}\}
\end{equation}
and
\begin{equation}\nonumber
\lambda_0:=\sup\,\Lambda_0.
\end{equation}

\noindent \emph{Step 2: we have that $\lambda_0=0$.} We argue by
contradiction and suppose that $\lambda_0<0$. By continuity we know
that $\hat{u}_i\leq \hat{u}_{i,\lambda_0}$ in $\Sigma_{\lambda_0}\setminus
R_{\lambda_0}(\Gamma^* \cup \{0\})$ for every $i=1,...,m$. By the strong maximum principle we deduce that $\hat{u}_i < \hat{u}_{i,\lambda_0}$ in $\Sigma_{\lambda_0}\setminus R_{\lambda_0}(\Gamma^* \cup \{0\})$ for every $i=1,...,m$. Indeed, $\hat{u}_i= \hat{u}_{i,\lambda_0}$ in $\Sigma_{\lambda_0}\setminus R_{\lambda_0}(\Gamma^* \cup \{0\})$) is not possible if $\lambda_0<0$, since in this case each $\hat{u}_i$ would be singular somewhere on $R_{\lambda_0}(\Gamma^* \cup \{0\})$. Now, for some $\bar\tau>0$, that will be fixed later on, and for any
$0<\tau<\bar\tau$ we show that $\hat{u}_i \leq \hat{u}_{i,\lambda_0+\tau}$ in
$\Sigma_{\lambda_0+\tau}\setminus R_{\lambda_0+\tau}(\Gamma^* \cup
\{0\})$ obtaining a contradiction with the definition of $\lambda_0$
and thus proving the claim. To this end we are going to show that,
for every $ \delta>0$ there are $ \bar{\tau}(\delta, \lambda_0)>0 $
and a compact set $K$ (depending on $\delta$ and $\lambda_0$) such
that
\[
K \subset \Sigma_{\lambda}\setminus R_{\lambda}(\Gamma^* \cup \{0\}), \qquad
\int_{\Sigma_{\lambda}\setminus K}\, \hat{u}_i^{2^*} < \delta, \qquad \forall \,  \lambda \in [\lambda_0, \lambda_0 + \bar{\tau}] \text{ and $i=1,...,m$}.
\]
To see this, we note that for every  every $ \delta>0$ there are $\tau_1(\delta, \lambda_0)>0 $ and a compact set $K$ (depending on
$\delta $ and $ \lambda_0$) such that $\displaystyle
\int_{\Sigma_{\lambda_0}\setminus K}\,\hat{u}_i^{2^*} < \frac{\delta}{2}$ for every $i=1,...,m$ and $K \subset \Sigma_{\lambda}\setminus R_{\lambda}(\Gamma^* \cup
\{0\})$ for every $ \lambda \in [\lambda_0, \lambda_0 + \tau_1].$
Consequently $\hat{u}_i$ and $\hat{u}_{i,\lambda}$ are well defined on $K$ for every
$ \lambda \in [\lambda_0, \lambda_0 + \tau_1].$ Hence, by the
uniform continuity of the functions $ g_i(x,\lambda) := \hat{u}_i(x) -
\hat{u}_i(2\lambda-x_1,x^{'}) $ on the compact set $K \times [\lambda_0,
\lambda_0 + \tau_1]$ we can ensure that  $K \subset
\Sigma_{{\lambda_0+\tau}} \setminus R_{{\lambda_0+\tau}}(\Gamma^*
\cup \{0\})$ and $\hat{u}_i< \hat{u}_{i,\lambda_0+\tau}$ in $K$ for any $0 \le \tau<
\tau_2$, for some $\tau_2= \tau(\delta,\lambda_0) \in (0, \tau_1)$.
Clearly we can also assume that $\tau_2< \frac{\vert \lambda_0
\vert}{4}.$ Finally, since $\hat{u}_i^{2^*} \in L^1(\Sigma_{\lambda_0 +
\frac{\vert \lambda_0 \vert}{4}})$ and $\displaystyle
\int_{\Sigma_{\lambda_0}\setminus K}\,\hat{u}_i^{2^*} < \frac{\delta}{2}$ for each $i=1,...,m$, we obtain the existence of $\bar{\tau} \in (0, \tau_2)$ such that
$\displaystyle \int_{\Sigma_{\lambda}\setminus K} \, \hat{u}_i^{2^*} <
\delta$ for all $ \lambda \in [\lambda_0, \lambda_0 + \bar{\tau}]$ and $i=1,...,m$.

Now we repeat verbatim the arguments used in the proof of Lemma \ref{E:stimgradcrit} but using the test functions
$$\varphi_i \,:= \begin{cases}
\, w_{i,\lambda_0+\tau}^+\psi_\varepsilon^2\eta_R^2 \, & \text{in}\quad\Sigma_{\lambda_0+\tau}  \\
0 &  \text{in}\quad \R^n \setminus \Sigma_{\lambda_0+\tau}.
\end{cases}$$

Thus we recover the last inequality in \eqref{E:aieiei}, which
immediately gives, for any $0 \le \tau<\bar\tau$
\begin{equation}\label{E:finalstepcritrn}
\begin{split}
&\sum_{i=1}^m \int_{\Sigma_{\lambda_0+\tau}\setminus K}|\nabla w_{i,\lambda_0+\nu}^+|^2\\ 
&\leq  \frac{n+2}{n-2} \sum_{i,j=1}^m a_{ij} \left[2\delta_{ij} C_{i,S}^2 \|\hat{u}_i\|_{L^{2^*}(\Sigma_{\lambda_0+\tau}\setminus K)}^{2^*-2}  + (1-\delta_{ij}) C_{j,S}^2\| \hat{u}_j \|_{L^{2^*}(\Sigma_{\lambda_0+\tau}\setminus K)}^{2^*-2}\right] \int_{\Sigma_{\lambda_0+\tau}\setminus K} |\nabla w_{i,\lambda}^+|^2\,dx \\
\end{split}
\end{equation}
since $ w_{i,\lambda_0+\tau}^+$ and $\nabla w_{i,\lambda_0+\tau}^+$ are zero in a neighborhood of $K$, by the above construction for every $i=1,...,m$.
Now we fix $\delta >0$ such that for every $i=1,...,m$ we have
\[
\frac{n+2}{n-2} \sum_{j=1}^m a_{ij} \left[2\delta_{ij} C_{i,S}^2 \|\hat{u}_i\|_{L^{2^*}(\Sigma_{\lambda_0+\tau}\setminus K)}^{2^*-2}  + (1-\delta_{ij}) C_{j,S}^2\| \hat{u}_j \|_{L^{2^*}(\Sigma_{\lambda_0+\tau}\setminus K)}^{2^*-2}\right]< \frac{1}{2}, \quad \forall \,\, 0 \le \tau<\bar\tau
\]
which plugged into \eqref{E:finalstepcritrn} implies that
$\displaystyle \int_{\Sigma_{\lambda_0+\tau}\setminus K}|\nabla
w_{i,\lambda_0+\tau}^+|^2\,dx = 0$ for every $0 \le \tau<\bar\tau$ and $i=1,...,m$.
Hence $\displaystyle \int_{\Sigma_{\lambda_0+\tau}}|\nabla
w_{i,\lambda_0+\tau}^+|^2\,dx = 0$ for every $0 \le \tau<\bar\tau$,
since $\nabla w_{i,\lambda_0+\tau}^+$ are zero in a neighborhood of
$K$. The latter and Lemma \ref{E:stimgradcrit} imply that $
w_{i,\lambda_0+\tau}^+ =0 $ on $ \Sigma_{\lambda_0 +\tau}$ for every
$0 \le \tau<\bar\tau$ and $i=1,...,m$, thus $\hat{u}_i\leq \hat{u}_{i,\lambda_0+\tau}$ in $\Sigma_{\lambda_0+\tau}\setminus R_{\lambda_0+\tau}(\Gamma^* \cup
\{0\})$ for every $0 \le \tau<\bar\tau$ and $i=1,...,m$. Which proves the claim of
Step 2.

\noindent \emph{Step 3: conclusion.} The symmetry of the Kelvin
transform $(\hat{u}_1,...,\hat{u}_m)$ follows now performing the moving plane method in the opposite direction. The fact that every $\hat{u}_i$ is symmetric w.r.t. the
hyperplane $\{ x_1=0 \}$ implies the symmetry of the solution $(u_1,...,u_m)$ w.r.t. the hyperplane $\{ x_1=0 \}$. The last claim then follows by the invariance of the considered problem with respect to isometries
(translations and rotations).

\end{proof}

\section{Proof of Theorem \ref{E:main3}}\label{E:mainproof}

\begin{proof}[Proof of Theorem \ref{E:main3}]
	
    As we observed in the proof of Theorem \ref{E:main2}, thanks to a well-known result of Brezis and Kato  \cite{BK} and standard elliptic estimates (see also \cite{Stru}), the solution $(u,v)$ is smooth in $\R^n \setminus \Gamma$.
	Furthermore we recall that it is enough to prove the theorem for the special case in which {\textit {the origin does not belong to}} $\Gamma$.

	\noindent Under this assumption, we consider the map $K: \R^n \setminus \{0\} \longrightarrow \R^n  \setminus \{0\} $ defined by $K = K(x):=\frac{x}{|x|^2}$. Given  $(u,v)$ solution to \eqref{E:doublecriticalsingularequbdd},
	its Kelvin transform is given by
	\begin{equation}\label{E:kelv} \left(\hat{u}(x),\hat{v}(x)\right):=\left( \frac{1}{|x|^{n-2}} u\left(\frac{x}{|x|^2}\right), \frac{1}{|x|^{n-2}} v\left(\frac{x}{|x|^2}\right)\right) \quad
	x \in \R^n \setminus \{\Gamma^* \cup \{0\}\},
	\end{equation}
	where $\Gamma^* = K( \Gamma)$. It follows that $(\hat{u},\hat{v})$ weakly
	satisfies \eqref{E:doublecriticalsingularequbdd} in $\R^n
	\setminus \{\Gamma^* \cup \{0\}\}$ and that $\Gamma^*\subset
	\{x_1=0\}$ since, by assumption, $\Gamma\subset \{x_1=0\}$.
	Furthermore, we also have that $\Gamma^*$ is bounded (not
	necessarily closed) since we assumed that $0\notin\Gamma$.

	\noindent  Let us now fix some notations. We set
	\begin{equation}\label{E:eq:sn2hgjhsdgf} \nonumber
	\Sigma_\lambda=\{x\in\mathbb{R}^n\,:\,x_1 <\lambda\}\,.
	\end{equation}
	As above $x_\lambda=(2\lambda-x_1,x_2,\ldots,x_n)$ is the reflection
	of $x$ through the hyperplane $T_\lambda=\{x=(x_1,...,x_n)\in \R^n \
	| \ x_1=\lambda\}$. Finally we consider the Kelvin transform $(\hat{u}, \hat{v})$ of
	$(u,v)$ defined in \eqref{E:kelv} and we set
	\begin{equation}\label{E:mov} \nonumber
	\begin{split}
	\xi_{\lambda}(x)=\hat{u}(x)-\hat{u}_\lambda (x)= \hat{u}(x)- \hat{u}(x_\lambda),\\
	\zeta_{\lambda}(x)=\hat{v}(x)-\hat{v}_\lambda (x)= \hat{v}(x) - \hat{v}(x_\lambda).
	\end{split}
	\end{equation}
	Note that $(\hat{u}, \hat{v})$ weakly solves
	\begin{equation}\label{E:debil122bissj}
	\begin{array}{lr}
	\displaystyle \int_{\mathbb{R}^n} \nabla \hat{u} \nabla
	\varphi\,dx\,=\,\int_{\mathbb{R}^n} \hat{u}^{2^*-1}\varphi\,dx + \frac{\alpha}{2^*}\int_{\mathbb{R}^n} \hat{u}^{\alpha-1} \hat{v}^\beta	\varphi \,dx \qquad & \forall \varphi\in C^{1}_c(\mathbb{R}^n\setminus
	\Gamma^* \cup \{0\})\,. \\
	\\
	\displaystyle \int_{\mathbb{R}^n} \nabla \hat{v} \nabla
	\psi\,dx\,=\,\int_{\mathbb{R}^n} \hat{v}^{2^*-1}\psi\,dx + \frac{\beta}{2^*}\int_{\mathbb{R}^n} \hat{u}^\alpha \hat{v}^{\beta-1}	\psi \,dx \qquad &\forall \psi \in C^{1}_c(\mathbb{R}^n\setminus
	\Gamma^* \cup \{0\})\,.
	\end{array}
	\end{equation}
	and $(\hat{u}_\lambda, \hat{v}_\lambda)$ weakly solves
	\begin{equation}\label{E:debil122biss}
	\begin{array}{lr}
	\displaystyle \int_{\mathbb{R}^n} \nabla \hat{u}_\lambda \nabla
	\varphi\,dx\,=\,\int_{\mathbb{R}^n} \hat{u}_\lambda^{2^*-1}\varphi\,dx + \frac{\alpha}{2^*}\int_{\mathbb{R}^n} \hat{u}_\lambda^{\alpha-1} \hat{v}_\lambda^\beta	\varphi \,dx \qquad & \forall \varphi\in C^{1}_c(\mathbb{R}^n\setminus
	\Gamma^* \cup \{0\})\,. \\
	
	\\
	\displaystyle \int_{\mathbb{R}^n} \nabla \hat{v}_\lambda \nabla
	\psi\,dx\,=\,\int_{\mathbb{R}^n} \hat{v}_\lambda^{2^*-1}\psi\,dx + \frac{\beta}{2^*}\int_{\mathbb{R}^n} \hat{u}_\lambda^\alpha \hat{v}_\lambda^{\beta-1}	\psi \,dx \qquad & \forall \psi \in C^{1}_c(\mathbb{R}^n\setminus
	\Gamma^* \cup \{0\})\,.
	\end{array}
	\end{equation}
	
	The properties of the Kelvin transform, the fact that
	$0\notin\Gamma$ and the regularity of $u, v$ imply that $\vert \hat{u}(x)
	\vert \le C_u \vert x \vert^{2-N}$ and $\vert \hat{v}(x)
	\vert \le C_v \vert x \vert^{2-N} $ and for every $x \in \R^n$ such that
	$\vert x \vert \ge R$, where $C_u, C_v$ and $R$ are positive constants
	(depending on $u$ and $v$).  In particular, for every $ \lambda <0$, we have
	\begin{equation}\label{E:kelvinp} \nonumber
	\hat{u}, \hat{v}\in L^{2^*}(\Sigma_\lambda)\cap L^{\infty}(\Sigma_\lambda) \cap C^0 (\overline{\Sigma_\lambda}) \,.
	\end{equation}
	
	\begin{lem}\label{E:stimgrad} Under the assumption of Theorem
		\ref{E:main2}, for every $\lambda < 0$, we have that $ \xi_\lambda^+, \zeta_\lambda^+\in
		L^{2^*}(\Sigma_\lambda), \nabla \xi_\lambda^+,\nabla \zeta_\lambda^+  \in L^2(\Sigma_\lambda)
		$ and
		\begin{equation}\label{E:buonastima}
		\begin{split}
		\int_{\Sigma_\lambda}|\nabla \xi_\lambda^+|^2\,dx + \int_{\Sigma_\lambda}|\nabla \zeta_\lambda^+|^2\,dx 
		&\leq 2\frac{n+2}{n-2} \left[(1+\alpha) \|\hat{u}\|^{2^*}_{L^{2^*}(\Sigma_\lambda)} + (1+\beta) \|\hat{v}\|^{2^*}_{L^{2^*}(\Sigma_\lambda)}\right] \\
		\end{split}
		\end{equation}
	\end{lem}
	
	\begin{proof}
		We immediately see that $\xi_\lambda^+, \zeta_\lambda^+ \in L^{2^*}(\Sigma_\lambda),$
		since $0\leq \xi_\lambda^+\leq \hat{u} \in L^{2^*}(\Sigma_\lambda)$ and $0\leq \zeta_\lambda^+\leq \hat{v} \in L^{2^*}(\Sigma_\lambda)$. The
		rest of the proof follows the lines of the one of Lemma
		\ref{E:leaiuto2}. Arguing as in Section \ref{E:notations}, for every $ \varepsilon>0$,
		we can find a function $\psi_\varepsilon \in C^{0,1}(\R^N, [0,1])$
		such that
		$$\int_{\Sigma_\lambda} |\nabla \psi_\varepsilon|^2 < 4 \varepsilon$$
		and $\psi_\varepsilon = 0$ in an open neighborhood
		$\mathcal{B_{\varepsilon}}$ of $R_{\lambda}(\{\Gamma^* \cup
		\{0\}\})$, with $\mathcal{B_{\varepsilon}} \subset
		\Sigma_{\lambda}$.
		
		Fix $ R_0>0$ such that $R_{\lambda}(\{\Gamma^* \cup \{0\} )\subset
		B_{R_0} $ and, for every $ R > R_0$, let $\eta_R$ be a standard
		cut off function such that $ 0 \le \eta_R \le 1 $ on $ \R^n$,
		$\eta_R=1$ in $B_R$, $\eta_R=0$ outside $B_{2R}$ with
		$|\nabla\eta_R|\leq 2/R,$ and consider
		$$\Phi\,:= \begin{cases}
		\, \xi_\lambda^+\psi_\varepsilon^2\eta_R^2 \, & \text{in}\quad\Sigma_\lambda, \\
		0 &  \text{in}\quad \R^n \setminus  \Sigma_\lambda
		\end{cases} \quad \text{and} \quad
		\Psi\,:= \begin{cases}
		\, \zeta_\lambda^+\psi_\varepsilon^2\eta_R^2 \, & \text{in}\quad\Sigma_\lambda, \\
		0 &  \text{in}\quad \R^n \setminus  \Sigma_\lambda.
		\end{cases}$$
		Now, as in Lemma \ref{E:leaiuto} we see that $ \Phi, \Psi \in C^{0,1}_c(\R^n)$ with $supp(\Phi)$ and  $supp(\Psi)$ contained in $\overline {\Sigma_{\lambda} \cap B_{2R}} \setminus R_{\lambda}(\{\Gamma^* \cup \{0\}\})$ and
		\begin{equation}\label{E:gradvarphi-intero}
		\nabla \Phi = \psi_\varepsilon^2 \eta_R^2 \nabla \xi^+_\lambda +
		2 \xi_\lambda^+  (\psi_\varepsilon^2 \eta_R \nabla \eta_R + \psi_\varepsilon \eta_R^2  \nabla  \psi_\varepsilon),
		\end{equation}
		\begin{equation}\label{E:gradpsi-intero}
		\nabla \Psi = \psi_\varepsilon^2 \eta_R^2 \nabla \zeta^+_\lambda +
		2 \zeta_\lambda^+ (\psi_\varepsilon^2 \eta_R \nabla \eta_R + \psi_\varepsilon \eta_R^2  \nabla  \psi_\varepsilon).
		\end{equation}
		Therefore, by a standard density argument, we can use $\Phi$ and $\Psi$ as test functions respectively in \eqref{E:debil122bissj} and in \eqref{E:debil122biss} so that, subtracting we get
		\begin{equation}\label{E:djfsdjfbskfhasklfh}
		\begin{split}
		\int_{\Sigma_\lambda}|\nabla \xi^+_\lambda|^2 \psi_\varepsilon^2\eta_R^2\,dx&=
		-2\int_{\Sigma_\lambda}\nabla \xi^+_\lambda \nabla \psi_\varepsilon
		\xi_\lambda^+ \psi_\varepsilon\eta_R^2\,dx
		-2\int_{\Sigma_\lambda} \nabla \xi^+_\lambda \nabla \eta_R \xi_\lambda^+ \eta_R \psi_\varepsilon^2\,dx\\
		&+\int_{\Sigma_\lambda} (\hat{u}^{2^*-1}-\hat{u}_\lambda^{2^*-1}) \xi_\lambda^+ \psi_\varepsilon^2\eta_R^2\,dx\,\,\\
		&+ \frac{\alpha}{2^*}\int_{\Sigma_\lambda} (\hat{u}^{\alpha-1}\hat{v}^\beta-\hat{u}^{\alpha-1}_\lambda \hat{v}_\lambda^\beta)\xi_\lambda^+ \psi_\varepsilon^2\eta_R^2\,dx\,\,\\
		&=:\,I_1+I_2+I_3+I_4\,.
		\end{split}
		\end{equation}
		\begin{equation}\label{E:djfsdjfbskfhasklfhbis}
		\begin{split}
		\int_{\Sigma_\lambda}|\nabla \zeta^+_\lambda|^2 \psi_\varepsilon^2\eta_R^2\,dx&=
		-2\int_{\Sigma_\lambda}\nabla \zeta^+_\lambda \nabla \psi_\varepsilon
		\zeta_\lambda^+ \psi_\varepsilon\eta_R^2\,dx
		-2\int_{\Sigma_\lambda} \nabla \zeta^+_\lambda \nabla \eta_R \zeta_\lambda^+ \eta_R \psi_\varepsilon^2\,dx\\
		&+\int_{\Sigma_\lambda} (\hat{v}^{2^*-1}-\hat{v}_\lambda^{2^*-1}) \zeta_\lambda^+ \psi_\varepsilon^2\eta_R^2\,dx\,\,\\
		&+ \frac{\beta}{2^*}\int_{\Sigma_\lambda} (\hat{u}^\alpha\hat{v}^{\beta-1}-\hat{u}^\alpha_\lambda \hat{v}_\lambda^{\beta-1})\zeta_\lambda^+ \psi_\varepsilon^2\eta_R^2\,dx\,\\
		&=:\,E_1+E_2+E_3+E_4\,.
		\end{split}
		\end{equation}
		
		Exploiting also Young's inequality and recalling that
		$0\leq \xi_\lambda^+\leq \hat{u}$ and $0\leq \zeta_\lambda^+\leq \hat{v}$, we get that
		\begin{equation}\label{E:I1}
		\begin{split}
		|I_1|&\leq \frac{1}{4} \int_{\Sigma_\lambda}|\nabla \xi^+_\lambda|^2 \psi_\varepsilon^2\eta_R^2\,dx
		+4\int_{\Sigma_\lambda}| \nabla \psi_\varepsilon|^2(\xi_\lambda^+)^2\eta_R^2\,dx\\
		&\leq  \frac{1}{4} \int_{\Sigma_\lambda}|\nabla \xi^+_\lambda|^2 \psi_\varepsilon^2\eta_R^2\,dx
		+ 16 \varepsilon \|\hat{u}\|^2_{L^\infty(\Sigma_\lambda)} .\\
		\end{split}
		\end{equation}
		\begin{equation}\label{E:E1}
		\begin{split}
		|E_1|&\leq \frac{1}{4} \int_{\Sigma_\lambda}|\nabla \zeta^+_\lambda|^2 \psi_\varepsilon^2\eta_R^2\,dx
		+4\int_{\Sigma_\lambda}| \nabla \psi_\varepsilon|^2(\zeta_\lambda^+)^2\eta_R^2\,dx\\
		&\leq  \frac{1}{4} \int_{\Sigma_\lambda}|\nabla \zeta^+_\lambda|^2 \psi_\varepsilon^2\eta_R^2\,dx
		+ 16 \varepsilon \|\hat{v}\|^2_{L^\infty(\Sigma_\lambda)} .\\
		\end{split}
		\end{equation}
		
		Furthermore we have that
		\begin{equation}\label{E:I2}
		\begin{split}
		|I_2|&\leq \frac{1}{4} \int_{\Sigma_\lambda}|\nabla \xi^+_\lambda|^2 \psi_\varepsilon^2\eta_R^2\,dx
		+4\int_{\Sigma_\lambda\cap(B_{2R}\setminus B_{R})}|\nabla \eta_R|^2(\xi_\lambda^+)^2 \psi_\varepsilon^2\,dx\\
		&\leq  \frac{1}{4} \int_{\Sigma_\lambda}|\nabla \xi^+_\lambda|^2 \psi_\varepsilon^2 \eta_R^2\,dx \\
		&+ 4\left(\int_{\Sigma_\lambda\cap(B_{2R}\setminus B_{R})}|\nabla
		\eta_R|^n\,dx\right)^{\frac{2}{n}}
		\left(\int_{\Sigma_\lambda\cap(B_{2R}\setminus B_{R})}\hat{u}^{2^*}\,dx\right)^{\frac{n-2}{n}}\\
		&\leq  \frac{1}{4} \int_{\Sigma_\lambda}|\nabla \xi^+_\lambda|^2 \psi_\varepsilon^2\eta_R^2\,dx\,+\,
		c(n) \left(\int_{\Sigma_\lambda\cap(B_{2R}\setminus B_{R})} \hat{u}^{2^*}\,dx\right)^{\frac{n-2}{n}}.
		\end{split}
		\end{equation}
		\begin{equation}\label{E:E2}
		\begin{split}
		|E_2|&\leq \frac{1}{4} \int_{\Sigma_\lambda}|\nabla \zeta^+_\lambda|^2 \psi_\varepsilon^2\eta_R^2\,dx
		+4\int_{\Sigma_\lambda\cap(B_{2R}\setminus B_{R})}|\nabla \eta_R|^2(\zeta_\lambda^+)^2 \psi_\varepsilon^2\,dx\\
		&\leq  \frac{1}{4} \int_{\Sigma_\lambda}|\nabla \zeta^+_\lambda|^2 \psi_\varepsilon^2 \eta_R^2\,dx \\
		&+ 4\left(\int_{\Sigma_\lambda\cap(B_{2R}\setminus B_{R})}|\nabla
		\eta_R|^n\,dx\right)^{\frac{2}{n}}
		\left(\int_{\Sigma_\lambda\cap(B_{2R}\setminus B_{R})}\hat{v}^{2^*}\,dx\right)^{\frac{n-2}{n}}\\
		&\leq  \frac{1}{4} \int_{\Sigma_\lambda}|\nabla \zeta^+_\lambda|^2 \psi_\varepsilon^2\eta_R^2\,dx\,+\,
		c(n) \left(\int_{\Sigma_\lambda\cap(B_{2R}\setminus B_{R})} \hat{v}^{2^*}\,dx\right)^{\frac{n-2}{n}}
		\end{split}
		\end{equation}
		where $c(n)$ is a positive constant depending only on the dimension $n$.
		
		Let us now estimate $I_3$ and $E_3$. Since $\hat{u}(x), \hat{u}_\lambda(x), \hat{v}(x), \hat{v}_\lambda(x)>0$, by the
		convexity of $ t \to t^{2^*-1},$ for $ t >0$, we obtain
		$$\hat{u}^{2^*-1}(x)-\hat{u}_\lambda^{2^*-1}(x) \le  \frac{n+2}{n-2}
		\hat{u}_\lambda^{2^*-2}(x) (\hat{u}(x) - \hat{u}_\lambda(x))$$ and $$\hat{v}^{2^*-1}(x)-\hat{v}_\lambda^{2^*-1}(x) \le  \frac{n+2}{n-2}
		\hat{v}_\lambda^{2^*-2}(x) (\hat{v}(x) - \hat{v}_\lambda(x)),$$ for every $x \in
		\Sigma_{\lambda}$. Thus, by making use of the monotonicity of $ t
		\to t^{2^*-2}$, for $ t >0$ and the definition of $\xi_\lambda^+$ and $\zeta_\lambda^+$ we
		get  $$(\hat{u}^{2^*-1}-\hat{u}_\lambda^{2^*-1})\xi_\lambda^+ \le  \frac{n+2}{n-2}
		\hat{u}_\lambda^{2^*-2}(\hat{u}-\hat{u}_\lambda) \xi_\lambda^+ \le \frac{n+2}{n-2}
		\hat{u}^{2^*-2}(\xi_\lambda^+)^2$$ and $$(\hat{v}^{2^*-1}-\hat{v}_\lambda^{2^*-1})\zeta_\lambda^+ \le  \frac{n+2}{n-2}
		\hat{v}_\lambda^{2^*-2}(\hat{v}-\hat{v}_\lambda) \zeta_\lambda^+ \le \frac{n+2}{n-2}
		\hat{v}^{2^*-2}(\zeta_\lambda^+)^2.$$ Therefore
		\begin{equation}\label{E:I3}
		\begin{split}
		|I_3|& \leq
		\frac{n+2}{n-2} \int_{\Sigma_\lambda}
		\hat{u}^{2^*-2}(\xi_\lambda^+)^2 \psi_\varepsilon^2\eta_R^2\,dx\,\\
		&\leq \frac{n+2}{n-2} \int_{\Sigma_\lambda} \hat{u}^{2^*-2}\hat{u}^2 dx\,=
		\frac{n+2}{n-2} \int_{\Sigma_\lambda} \hat{u}^{2^*}\,dx = \frac{n+2}{n-2} \|\hat{u}\|^{2^*}_{L^{2^*}(\Sigma_\lambda)}\\
		\end{split}
		\end{equation}
		\begin{equation}\label{E:E3}
		\begin{split}
		|E_3|& \leq
		\frac{n+2}{n-2} \int_{\Sigma_\lambda}
		\hat{v}^{2^*-2}(\zeta_\lambda^+)^2 \psi_\varepsilon^2\eta_R^2\,dx\,\\
		&\leq \frac{n+2}{n-2} \int_{\Sigma_\lambda} \hat{v}^{2^*-2}\hat{v}^2 dx\,=
		\frac{n+2}{n-2} \int_{\Sigma_\lambda} \hat{v}^{2^*}\,dx = \frac{n+2}{n-2} \|\hat{v}\|^{2^*}_{L^{2^*}(\Sigma_\lambda)}\\
		\end{split}
		\end{equation}
		where we also used that $0\leq \xi_\lambda^+\leq \hat{u}$ and $0\leq \zeta_\lambda^+\leq \hat{v}$.
		
		Finally we have to estimate $I_4$ and $E_4$. Since $\hat{u}(x), \hat{u}_\lambda(x), \hat{v}(x), \hat{v}_\lambda(x)>0$, by the
		convexity of the functions $t \to t^{\alpha}, t \to t^{\alpha-1}, t \to t^{\beta}, t \to t^{\beta-1}$ for $ t >0$, we obtain 	
		\begin{equation} \nonumber
		\begin{split}
		\hat{u}^{\alpha}(x)-\hat{u}_\lambda^{\alpha}(x) &\le  \alpha
		\hat{u}_\lambda^{\alpha-1}(x) (\hat{u}(x) - \hat{u}_\lambda(x)),\\
		\hat{u}^{\alpha-1}(x)-\hat{u}_\lambda^{\alpha-1}(x) &\le  (\alpha-1)
		\hat{u}_\lambda^{\alpha-2}(x) (\hat{u}(x) - \hat{u}_\lambda(x)),\\ \hat{v}^{\beta}(x)-\hat{v}_\lambda^{\beta}(x) &\le  \beta
		\hat{v}_\lambda^{\beta-1}(x) (\hat{v}(x) - \hat{v}_\lambda(x)), \\ \hat{v}^{\beta-1}(x)-\hat{v}_\lambda^{\beta-1}(x) &\le  (\beta-1)
		\hat{v}_\lambda^{\beta-2}(x) (\hat{v}(x) - \hat{v}_\lambda(x)),
		\end{split}
		\end{equation} 
		for every $x \in \Sigma_{\lambda}$. By the monotonicity of $t \to t^{\alpha}, t \to t^{\alpha-1}, t \to t^{\beta}, t \to t^{\beta-1}$ for $ t >0$ and the definition of $\xi_\lambda^+$ and $\zeta_\lambda^+$ we get 
		\begin{equation} \nonumber
		\begin{split} 
		(\hat{u}^{\alpha}(x)-\hat{u}_\lambda^{\alpha}(x))\xi_\lambda^+ &\le  \alpha
		\hat{u}_\lambda^{\alpha-2}(\hat{u}-\hat{u}_\lambda) \xi_\lambda^+ \le \alpha
		\hat{u}^{\alpha-2}(\xi_\lambda^+)^2,\\
		(\hat{u}^{\alpha-1}(x)-\hat{u}_\lambda^{\alpha-1}(x))\xi_\lambda^+ &\le  (\alpha-1)\hat{u}_\lambda^{\alpha-2}(\hat{u}-\hat{u}_\lambda) \xi_\lambda^+ \le (\alpha-1) \hat{u}^{\alpha-2}(\xi_\lambda^+)^2,\\ (\hat{v}^{\beta}-\hat{v}_\lambda^{\beta})\zeta_\lambda^+ &\le  \beta
		\hat{v}_\lambda^{\beta-2}(\hat{v}-\hat{v}_\lambda) \zeta_\lambda^+ \le \beta
		\hat{v}^{\beta-2}(\zeta_\lambda^+)^2,\\  (\hat{v}^{\beta-1}-\hat{v}_\lambda^{\beta-1})\zeta_\lambda^+ &\le  (\beta-1)
		\hat{v}_\lambda^{\beta-2}(\hat{v}-\hat{v}_\lambda) \zeta_\lambda^+ \le (\beta-1)
		\hat{v}^{\beta-2}(\zeta_\lambda^+)^2.
		\end{split}
		\end{equation}  
		Now, having in mind all these estimates, we need a fine analysis in view of the cooperativity of the system. Since $\alpha+\beta=2^*=\frac{2n}{n-2}$ and $\alpha,\beta \geq 2$ we have to split

		\begin{equation}\label{E:I_4}
		\begin{split}
		|I_4| &\leq \frac{\alpha}{2^*}\int_{\Sigma_\lambda} |\hat{u}^{\alpha-1}\hat{v}^\beta-\hat{u}^{\alpha-1} \hat{v}_\lambda^\beta|\xi_\lambda^+ \psi_\varepsilon^2\eta_R^2\,dx\,+ \frac{\alpha}{2^*}\int_{\Sigma_\lambda} |\hat{u}^{\alpha-1}\hat{v}_\lambda^\beta-\hat{u}^{\alpha-1}_\lambda \hat{v}_\lambda^\beta|\xi_\lambda^+ \psi_\varepsilon^2\eta_R^2\,dx\,\\
		& \leq \frac{\alpha \beta}{2^*}\int_{\Sigma_\lambda} \hat{u}^{\alpha-1} \hat{v}_\lambda^{\beta-1} \xi_\lambda^+ \zeta_\lambda^+  \psi_\varepsilon^2\eta_R^2\,dx\,+ \frac{\alpha (\alpha-1)}{2^*}\int_{\Sigma_\lambda} \hat{u}_\lambda^{\alpha-2} \hat{v}_\lambda^{\beta}(\xi_\lambda^+)^2 \psi_\varepsilon^2\eta_R^2\,dx\,\\
		& \leq \frac{\alpha \beta}{2^*}\int_{\Sigma_\lambda} \hat{u}^{\alpha-1} \hat{v}^{\beta-1}\hat{u} \hat{v}\psi_\varepsilon^2\eta_R^2\,dx\,+ \frac{\alpha (\alpha-1)}{2^*}\int_{\Sigma_\lambda} \hat{u}^{\alpha-2} \hat{v}^{\beta}\hat{u}^2 \psi_\varepsilon^2\eta_R^2\,dx\,\\
		&\leq\frac{\alpha \beta}{2^*}\int_{\Sigma_\lambda} \hat{u}^{\alpha} \hat{v}^{\beta}\,dx\,+ \frac{\alpha (\alpha-1)}{2^*}\int_{\Sigma_\lambda} \hat{u}^{\alpha} \hat{v}^{\beta}\,dx\,\\
		&=\frac{\alpha (2^*-1)}{2^*}\int_{\Sigma_\lambda} \hat{u}^{\alpha} \hat{v}^{\beta}\,dx,
		\end{split}
		\end{equation}

		\begin{equation}\label{E:E_4}
		\begin{split}
		|E_4| &\leq \frac{\beta}{2^*}\int_{\Sigma_\lambda} |\hat{u}^{\alpha}\hat{v}^{\beta-1}-\hat{u}_\lambda^{\alpha} \hat{v}^{\beta-1}|\zeta_\lambda^+ \psi_\varepsilon^2\eta_R^2\,dx\,+ \frac{\beta}{2^*}\int_{\Sigma_\lambda} |\hat{u}_\lambda^{\alpha}\hat{v}^{\beta-1}-\hat{u}^{\alpha}_\lambda \hat{v}_\lambda^{\beta-1}|\zeta_\lambda^+ \psi_\varepsilon^2\eta_R^2\,dx\,\\
		& \leq \frac{\alpha \beta}{2^*}\int_{\Sigma_\lambda} \hat{u}_\lambda^{\alpha-1} \hat{v}^{\beta-1} \xi_\lambda^+ \zeta_\lambda^+  \psi_\varepsilon^2\eta_R^2\,dx\,+ \frac{\beta (\beta-1)}{2^*}\int_{\Sigma_\lambda} \hat{u}_\lambda^{\alpha} \hat{v}_\lambda^{\beta-2}(\zeta_\lambda^+)^2 \psi_\varepsilon^2\eta_R^2\,dx\,\\
    	& \leq \frac{\alpha \beta}{2^*}\int_{\Sigma_\lambda} \hat{u}^{\alpha-1} \hat{v}^{\beta-1}\hat{u} \hat{v}\psi_\varepsilon^2\eta_R^2\,dx\,+ \frac{\beta (\beta-1)}{2^*}\int_{\Sigma_\lambda} \hat{u}^{\alpha} \hat{v}^{\beta-2}\hat{v}^2 \psi_\varepsilon^2\eta_R^2\,dx\,\\
    	&\leq\frac{\alpha \beta}{2^*}\int_{\Sigma_\lambda} \hat{u}^{\alpha} \hat{v}^{\beta}\,dx\,+ \frac{\alpha (\alpha-1)}{2^*}\int_{\Sigma_\lambda} \hat{u}^{\alpha} \hat{v}^{\beta}\,dx\,\\
    	&=  \frac{\beta (2^*-1)}{2^*} \int_{\Sigma_\lambda} \hat{u}^{\alpha} \hat{v}^{\beta}\,dx.
	    \end{split}
		\end{equation}

		Hence, by applying H\"older inequality with exponents $\displaystyle \left(\frac{\alpha}{2^*}, \frac{\beta}{2^*}\right)$ it follows that 
		\begin{equation} \nonumber
		|I_4|+|E_4| \leq (2^*-1) \int_{\Sigma_\lambda} \hat{u}^{\alpha} \hat{v}^{\beta}\,dx\, \leq (2^*-1) \| \hat{u}\|_{L^{2^*}(\Sigma_\lambda)}^\alpha \| \hat{v}\|_{L^{2^*}(\Sigma_\lambda)}^\beta.
		\end{equation}

		Taking into account the estimates on $I_1$, $I_2$, $I_3$, $I_4$, $E_1$, $E_2$, $E_3$ and $E_4$, by adding \eqref{E:djfsdjfbskfhasklfh} and \eqref{E:djfsdjfbskfhasklfhbis}, we deduce that5		
		\begin{equation}\label{E:primastima} \nonumber
		\begin{split}
		\int_{\Sigma_\lambda}|\nabla \xi^+_\lambda|^2 \psi_\varepsilon^2\eta_R^2\,dx + \int_{\Sigma_\lambda}|\nabla \zeta^+_\lambda|^2 \psi_\varepsilon^2\eta_R^2\,dx
		&\leq 32 \varepsilon \left(\|\hat{u}\|^2_{L^\infty(\Sigma_\lambda)} + \|\hat{v}\|^2_{L^\infty(\Sigma_\lambda)}\right)\\
		&+ 2c(n) \left(\int_{\Sigma_\lambda\cap(B_{2R}\setminus B_{R})} \hat{u}^{2^*}\,dx\right)^{\frac{n-2}{n}}\\
		&+ 2c(n) \left(\int_{\Sigma_\lambda\cap(B_{2R}\setminus B_{R})} \hat{v}^{2^*}\,dx\right)^{\frac{n-2}{n}}\\
		&+ 2\frac{n+2}{n-2} \left(\|\hat{u}\|^{2^*}_{L^{2^*}(\Sigma_\lambda)} + \|\hat{v}\|^{2^*}_{L^{2^*}(\Sigma_\lambda)}\right) \\
		&+ 2  (2^*-1) \| \hat{u}\|_{L^{2^*}(\Sigma_\lambda)}^\alpha \| \hat{v}\|_{L^{2^*}(\Sigma_\lambda)}^\beta .
		\end{split}
		\end{equation}
		
		By Fatou Lemma, as $\varepsilon$ tends to zero and $R$ tends to
		infinity,  we deduce that $\nabla \xi_\lambda^+, \nabla \zeta_\lambda^+ 
		\in L^2(\Sigma_{\lambda})$.  We also note that $ \Phi \to
		\xi_\lambda^+$ and $\Psi \to \zeta_\lambda^+$ in $L^{2^*}(\Sigma_\lambda)$, by definition of $\Phi$ and $\Psi$, and that $\nabla \Phi \to \nabla \xi_\lambda^+$ and $\nabla \Psi \to \nabla \zeta_\lambda^+$ in $L^2(\Sigma_{\lambda})$, by	\eqref{E:gradvarphi-intero}, \eqref{E:gradpsi-intero} and the fact that $\xi_\lambda^+, \zeta_\lambda^+ \in
		L^{2^*}(\Sigma_\lambda)$.  Therefore 
		\begin{equation}\label{E:thesis}
		\begin{split}
		\int_{\Sigma_\lambda}|\nabla \xi^+_\lambda|^2 \,dx + \int_{\Sigma_\lambda}|\nabla \zeta^+_\lambda|^2 \,dx &\leq  2\frac{n+2}{n-2} \left(\|\hat{u}\|^{2^*}_{L^{2^*}(\Sigma_\lambda)} + \|\hat{v}\|^{2^*}_{L^{2^*}(\Sigma_\lambda)}\right) \\
		&+ 2  (2^*-1) \| \hat{u}\|_{L^{2^*}(\Sigma_\lambda)}^\alpha \| \hat{v}\|_{L^{2^*}(\Sigma_\lambda)}^\beta.
		\end{split}
		\end{equation}
	    Exploiting Young inequality in the right hand side of \eqref{E:thesis}, with conjugate exponents $\displaystyle \left(\frac{\alpha}{2^*}, \frac{\beta}{2^*}\right)$, we obtain \eqref{E:buonastima}.

	\end{proof}
	
	We can now complete the proof of Theorem \ref{E:main3}. As for the proof of Theorem \ref{E:semilinearSym} and Theorem \ref{E:main2},  we split the proof into three steps and we start with
	
	\noindent \emph{Step 1: there exists $M>1$ such that $\hat{u} \leq \hat{u}_\lambda$ and $\hat{v} \leq \hat{v}_\lambda$ in $\Sigma_\lambda\setminus R_\lambda(\Gamma^* \cup \{0\})$, for all $\lambda< -M$.}
	
	Arguing as in the proof of Lemma \ref{E:stimgrad} and using the same notations and the same construction
	for $\psi_\varepsilon$, $\eta_R$, $\varphi$ and $\psi$, we get
\begin{equation}\label{E:djfsdjfbskfhasklfh2} \nonumber
\begin{split}
\int_{\Sigma_\lambda}|\nabla \xi^+_\lambda|^2 \psi_\varepsilon^2\eta_R^2\,dx&=
-2\int_{\Sigma_\lambda}\nabla \xi^+_\lambda \nabla \psi_\varepsilon
\xi_\lambda^+ \psi_\varepsilon\eta_R^2\,dx
-2\int_{\Sigma_\lambda} \nabla \xi^+_\lambda \nabla \eta_R \xi_\lambda^+ \eta_R \psi_\varepsilon^2\,dx\\
&+\int_{\Sigma_\lambda} (\hat{u}^{2^*-1}-\hat{u}_\lambda^{2^*-1}) \xi_\lambda^+ \psi_\varepsilon^2\eta_R^2\,dx\,\,\\
&+ \frac{\alpha}{2^*}\int_{\Sigma_\lambda} (\hat{u}^{\alpha-1}\hat{v}^\beta-\hat{u}^{\alpha-1}_\lambda \hat{v}_\lambda^\beta)\xi_\lambda^+ \psi_\varepsilon^2\eta_R^2\,dx\,\,\\
&=:\,I_1+I_2+I_3+I_4\,.
\end{split}
\end{equation}
\begin{equation}\label{E:djfsdjfbskfhasklfhbis2} \nonumber
\begin{split}
\int_{\Sigma_\lambda}|\nabla \zeta^+_\lambda|^2 \psi_\varepsilon^2\eta_R^2\,dx&=
-2\int_{\Sigma_\lambda}\nabla \zeta^+_\lambda \nabla \psi_\varepsilon
\zeta_\lambda^+ \psi_\varepsilon\eta_R^2\,dx
-2\int_{\Sigma_\lambda} \nabla \zeta^+_\lambda \nabla \eta_R \zeta_\lambda^+ \eta_R \psi_\varepsilon^2\,dx\\
&+\int_{\Sigma_\lambda} (\hat{v}^{2^*-1}-\hat{v}_\lambda^{2^*-1}) \zeta_\lambda^+ \psi_\varepsilon^2\eta_R^2\,dx\,\,\\
&+ \frac{\beta}{2^*}\int_{\Sigma_\lambda} (\hat{u}^\alpha\hat{v}^{\beta-1}-\hat{u}^\alpha_\lambda \hat{v}_\lambda^{\beta-1})\zeta_\lambda^+ \psi_\varepsilon^2\eta_R^2\,dx\,\\
&=:\,E_1+E_2+E_3+E_4\,.
\end{split}
\end{equation}
	where $I_1, \ E_1, \ I_2, \ E_2, \ I_3, \ E_3, \ I_4$ and $E_4$ can be estimated exactly as in \eqref{E:I1}, \eqref{E:E1}, \eqref{E:I2}, \eqref{E:E2}, \eqref{E:I3}, \eqref{E:E3}, \eqref{E:I_4} and \eqref{E:E_4}. The latter yield
	\begin{equation}\nonumber
	\begin{split}
	\int_{\Sigma_\lambda}\left(|\nabla \xi^+_\lambda|^2 +|\nabla \zeta^+_\lambda|^2 \right)  &\psi_\varepsilon^2\eta_R^2 \,dx 
	\leq 32 \varepsilon \left(\|\hat{u}\|^2_{L^\infty(\Sigma_\lambda)} + \|\hat{v}\|^2_{L^\infty(\Sigma_\lambda)}\right) \\
	&+ 2c(n) \left(\int_{\Sigma_\lambda\cap(B_{2R}\setminus B_{R})} \hat{u}^{2^*}\,dx\right)^{\frac{2}{2^*}}\\
	&+ 2c(n) \left(\int_{\Sigma_\lambda\cap(B_{2R}\setminus B_{R})} \hat{v}^{2^*}\,dx\right)^{\frac{2}{2^*}} + 2 \frac{n+2}{n-2} \int_{\Sigma_\lambda}
	\hat{u}^{2^*-2}(\xi_\lambda^+)^2 \psi_\varepsilon^2\eta_R^2\,dx\, \\
	&+ 2 \frac{n+2}{n-2} \int_{\Sigma_\lambda}
	\hat{v}^{2^*-2}(\zeta_\lambda^+)^2 \psi_\varepsilon^2\eta_R^2\,dx\, + 4\frac{\alpha \beta}{2^*}\int_{\Sigma_\lambda} \hat{u}^{\alpha-1} \hat{v}^{\beta-1} \xi_\lambda^+ \zeta_\lambda^+  \psi_\varepsilon^2\eta_R^2\,dx\, \\
	&+ \frac{\alpha (\alpha-1)}{2^*}\int_{\Sigma_\lambda} \hat{u}^{\alpha-2} \hat{v}^{\beta}(\xi_\lambda^+)^2 \psi_\varepsilon^2\eta_R^2\,dx\,\\
	&+ \frac{\beta (\beta-1)}{2^*}\int_{\Sigma_\lambda} \hat{u}^{\alpha} \hat{v}^{\beta-2}(\zeta_\lambda^+)^2 \psi_\varepsilon^2\eta_R^2\,dx.\\
	\end{split}
	\end{equation}
	Passing to the limit in the latter, as $\varepsilon$ tends to zero and $R$ tends to infinity, we obtain
	\begin{equation} \nonumber
	\begin{split}
	\int_{\Sigma_\lambda}|\nabla \xi^+_\lambda|^2 \,dx + \int_{\Sigma_\lambda}|\nabla \zeta^+_\lambda|^2 \,dx &\leq 2 \frac{n+2}{n-2} \left(\int_{\Sigma_\lambda}
	\hat{u}^{2^*-2}(\xi_\lambda^+)^2 \,dx\, + \int_{\Sigma_\lambda} \hat{v}^{2^*-2}(\zeta_\lambda^+)^2 \,dx\, \right)\\  
	& + 4\frac{\alpha \beta}{2^*}\int_{\Sigma_\lambda} \hat{u}^{\alpha-1} \hat{v}^{\beta-1} \xi_\lambda^+ \zeta_\lambda^+ \,dx\,\\
	&+ \frac{\alpha (\alpha-1)}{2^*}\int_{\Sigma_\lambda} \hat{u}^{\alpha-2} \hat{v}^{\beta}(\xi_\lambda^+)^2 \,dx\,\\
	&+ \frac{\beta (\beta-1)}{2^*}\int_{\Sigma_\lambda} \hat{u}^{\alpha} \hat{v}^{\beta-2}(\zeta_\lambda^+)^2 \,dx\, < +\infty.
	\end{split}
	\end{equation}
	which combined with Young inequality gives
	\begin{equation}\label{E:kushfkusfksf}
	\begin{split}
	\int_{\Sigma_\lambda}|\nabla \xi^+_\lambda|^2 \,dx + \int_{\Sigma_\lambda}|\nabla \zeta^+_\lambda|^2 \,dx &\leq 2 \frac{n+2}{n-2} \left(\int_{\Sigma_\lambda}
	\hat{u}^{2^*-2}(\xi_\lambda^+)^2 \,dx\, + \int_{\Sigma_\lambda} \hat{v}^{2^*-2}(\zeta_\lambda^+)^2 \,dx\, \right)\\  
	&+ \frac{\alpha (2^*+\beta-1)}{2^*}\int_{\Sigma_\lambda} \hat{u}^{\alpha-2} \hat{v}^{\beta}(\xi_\lambda^+)^2 \,dx\,\\
	&+ \frac{\beta (2^*+\beta-1)}{2^*}\int_{\Sigma_\lambda} \hat{u}^{\alpha} \hat{v}^{\beta-2}(\zeta_\lambda^+)^2 \,dx\,\\
	&=:A_1+A_2+A_3.
	\end{split}
	\end{equation}
	
	Exploiting H\"older inequality with conjugate exponents $\displaystyle \left( \frac{2^*}{2^*-2}, \frac{2^*}{2}\right)$ we obtain
	
	\begin{equation} \label{E:A_1}
    \begin{split}
    |A_1| \leq 2 \frac{n+2}{n-2} \left[\left(\int_{\Sigma_\lambda}
    \hat{u}^{2^*} \, dx\right)^\frac{2}{n} \left(\int_{\Sigma_\lambda} (\xi_\lambda^+)^{2^*} \,dx\right)^{\frac{2}{2^*}} +\left(\int_{\Sigma_\lambda} \hat{v}^{2^*} \, dx \right)^{\frac{2}{n}} \left(\int_{\Sigma_\lambda}(\zeta_\lambda^+)^{2^*} \,dx\right)^{\frac{2}{2^*}} \right].
    \end{split} 
	\end{equation}
	
	Exploiting H\"older inequality with conjugate exponents $\displaystyle \left( \frac{2^*}{\alpha-2}, \frac{2^*}{\beta}, \frac{2^*}{2}\right)$ \big(we note that if $\alpha=2$ we have $\beta=2$ and the conjugate exponents would be $\left(\frac{2^*}{2},\frac{2^*}{2}\right)$\big) we obtain
	\begin{equation} \label{E:A_2}
	\begin{split}
	|A_2| \leq  \frac{\alpha(2^*+\beta-1)}{2^*} \left(\int_{\Sigma_\lambda}
	\hat{u}^{2^*} \, dx\right)^\frac{\alpha-2}{2^*} \left(\int_{\Sigma_\lambda}
	\hat{v}^{2^*} \, dx\right)^\frac{\beta}{2^*} \left(\int_{\Sigma_\lambda} (\xi_\lambda^+)^{2^*} \,dx\right)^{\frac{2}{2^*}}.
	\end{split} 
	\end{equation}
	
	Exploiting H\"older inequality with conjugate exponents $\displaystyle \left( \frac{2^*}{\alpha}, \frac{2^*}{\beta-2}, \frac{2^*}{2}\right)$ \big(we note that if $\beta=2$ we have $\alpha=2$ and the conjugate exponents would be $\left(\frac{2^*}{2},\frac{2^*}{2}\right)$\big)we obtain
	\begin{equation} \label{E:A_3}
	\begin{split}
	|A_3| \leq \frac{\beta(2^*+\alpha-1)}{2^*}\left(\int_{\Sigma_\lambda}
	\hat{u}^{2^*} \, dx\right)^\frac{\alpha}{2^*} \left(\int_{\Sigma_\lambda}
	\hat{v}^{2^*} \, dx\right)^\frac{\beta-2}{2^*} \left(\int_{\Sigma_\lambda} (\zeta_\lambda^+)^{2^*} \,dx\right)^{\frac{2}{2^*}}. 
	\end{split} 
	\end{equation}
	
	Combining \eqref{E:A_1}, \eqref{E:A_2} and \eqref{E:A_3} and applying Sobolev inequality to \eqref{E:kushfkusfksf}
	\begin{equation}\label{E:finaldoublicrit} \nonumber
	\begin{split}
	\int_{\Sigma_\lambda}|\nabla \xi^+_\lambda|^2 \,dx + \int_{\Sigma_\lambda}|\nabla \zeta^+_\lambda|^2 \,dx &\leq C_1 \int_{\Sigma_\lambda}|\nabla \xi^+_\lambda|^2 \,dx + C_2\int_{\Sigma_\lambda}|\nabla \zeta^+_\lambda|^2 \,dx,
	\end{split}
	\end{equation}
	
	where $C_1:=\left[2 \frac{n+2}{n-2} \| \hat{u}\|^{\beta}_{L^{2^*}(\Sigma_\lambda)} + \frac{\alpha (2^*+\beta-1)}{2^*} \| \hat{v}\|^{\beta}_{L^{2^*}(\Sigma_\lambda)}\right] C^2_{u,S} \| \hat{u}\|^{\alpha-2}_{L^{2^*}(\Sigma_\lambda)}$, \\
	$C_2:= \left[2 \frac{n+2}{n-2} \| \hat{v}\|^{\alpha}_{L^{2^*}(\Sigma_\lambda)} +  \frac{\beta (2^*+\beta-1)}{2^*} \|\hat{u}\|^{\alpha}_{L^{2^*}(\Sigma_\lambda)} \right] C^2_{v,S} \| \hat{v}\|^{\beta-2}_{L^{2^*}(\Sigma_\lambda)}$, $C_{u,S}$ and $C_{v,S}$ are the Sobolev constants. Recalling that $\hat{u}, \hat{v} \in L^{2^*}(\Sigma_\lambda)$, we deduce the existence of $ M>1$ such that
	\[
	C_1:=\left[2 \frac{n+2}{n-2} \| \hat{u}\|^{\beta}_{L^{2^*}(\Sigma_\lambda)} + \frac{\alpha (2^*+\beta-1)}{2^*} \| \hat{v}\|^{\beta}_{L^{2^*}(\Sigma_\lambda)}\right] C^2_{u,S} \| \hat{u}\|^{\alpha-2}_{L^{2^*}(\Sigma_\lambda)}<1\,
	\]
	and 
	\[
	C_2:= \left[2 \frac{n+2}{n-2} \| \hat{v}\|^{\alpha}_{L^{2^*}(\Sigma_\lambda)} +  \frac{\beta (2^*+\beta-1)}{2^*} \|\hat{u}\|^{\alpha}_{L^{2^*}(\Sigma_\lambda)} \right] C^2_{v,S} \| \hat{v}\|^{\beta-2}_{L^{2^*}(\Sigma_\lambda)} < 1
	\]
	for every $\lambda < -M$. The latter and \eqref{E:kushfkusfksf} lead
	to
	\[
	\int_{\Sigma_\lambda}|\nabla \xi^+_\lambda|^2 \,dx\,=0 \text{\quad and \quad}  \int_{\Sigma_\lambda}|\nabla \zeta^+_\lambda|^2 \,dx\, =0.
	\]
	This implies that $\xi_\lambda^+=\zeta_\lambda^+=0$ by Lemma \ref{E:stimgrad} and the
	claim is proved.\\
	
	\noindent To proceed further we
	define
	\begin{equation}\nonumber
	\Lambda_0=\{\lambda<0 : \hat{u}\leq \hat{u}_{t}\,\,\,\text{and}\,\,\, \hat{v}\leq \hat{v}_{t}\,\,\,\text{in} \,\,\, \Sigma_t\setminus R_t(\Gamma^* \cup
	\{0\})\,\,\,\text{for all $t\in(-\infty,\lambda]$}\}
	\end{equation}
	and
	\begin{equation}\nonumber
	\lambda_0=\sup\,\Lambda_0.
	\end{equation}
	
	\noindent \emph{Step 2: we have that $\lambda_0=0$.} We argue by
	contradiction and suppose that $\lambda_0<0$. By continuity we know
	that $\hat{u}\leq \hat{u}_{\lambda_0}$ and $\hat{v}\leq \hat{v}_{\lambda_0}$ in $\Sigma_{\lambda_0}\setminus R_{\lambda_0}(\Gamma^* \cup \{0\})$. By the strong maximum principle we deduce that $\hat{u} < \hat{u}_{\lambda_0}$ and $\hat{v} < \hat{v}_{\lambda_0}$ in $\Sigma_{\lambda_0}\setminus R_{\lambda_0}(\Gamma^* \cup \{0\})$. Indeed, $\hat{u} = \hat{u}_{\lambda_0}$ and $\hat{v} = \hat{v}_{\lambda_0}$ in	$\Sigma_{\lambda_0}\setminus R_{\lambda_0}(\Gamma^* \cup \{0\})$) is not possible if $\lambda_0<0$, since in this case $\hat{u}$ and $\hat{v}$ would be singular somewhere on $R_{\lambda_0}(\Gamma^* \cup \{0\})$. Now, for	some $\bar\tau>0$, that will be fixed later on, and for any
	$0<\tau<\bar\tau$ we show that $\hat{u}\leq \hat{u}_{\lambda_0+\tau}$ and $\hat{v}\leq \hat{v}_{\lambda_0+\tau}$ in $\Sigma_{\lambda_0+\tau}\setminus R_{\lambda_0+\tau}(\Gamma^* \cup \{0\})$ obtaining a contradiction with the definition of $\lambda_0$ and proving thus the claim. To this end we recall that, repeating verbatim the argument used in the roof of Theorem \ref{E:main2}, it is possible to prove that for every $ \delta>0$ there are $ \bar{\tau}(\delta, \lambda_0)>0 $ and a compact set $K$ (depending on $\delta$ and $\lambda_0$) such that 
	\begin{equation} \label{E:compact}
	K \subset \Sigma_{\lambda}\setminus R_{\lambda}(\Gamma^* \cup \{0\}), \qquad
	\int_{\Sigma_{\lambda}\setminus K}\,\hat{u}^{2^*} < \delta \text{\; and \;} \int_{\Sigma_{\lambda}\setminus K}\,\hat{v}^{2^*} < \delta \qquad \forall \,  \lambda \in [\lambda_0, \lambda_0 + \bar{\tau}].
	\end{equation}
	
	Now we repeat verbatim the arguments used in the proof of Lemma \ref{E:stimgrad} but using the test function
	$$\Phi \,:= \begin{cases}
	\, \xi_{\lambda_0+\tau}^+\psi_\varepsilon^2\eta_R^2 \, & \text{in}\quad\Sigma_{\lambda_0+\tau}  \\
	0 &  \text{in}\quad \R^n \setminus \Sigma_{\lambda_0+\tau}.
	\end{cases} \quad \text{and} \quad \Psi\,:= \begin{cases}
	\, \zeta_{\lambda_0+\tau}^+\psi_\varepsilon^2\eta_R^2 \, & \text{in}\quad\Sigma_{\lambda_0+\tau}  \\
	0 &  \text{in}\quad \R^n \setminus \Sigma_{\lambda_0+\tau}.
	\end{cases}$$
	Thus we recover the first inequality in \eqref{E:kushfkusfksf}, and repeating verbatim the arguments used in \eqref{E:A_1}, \eqref{E:A_2} and \eqref{E:A_3} which
	immediately gives, for any $0 \le \tau<\bar\tau$
	\begin{equation}\label{E:finaldoublicritmov}
	\begin{split}
	\int_{\Sigma_{\lambda_0+\tau} \setminus K}|\nabla \xi^+_{\lambda_0+\tau}|^2 \,dx + \int_{\Sigma_{\lambda_0+\tau} \setminus K}|\nabla \zeta^+_{\lambda_0+\tau}|^2 \,dx &\leq C_1  C^2_{u,S} \| \hat{u}\|^{\alpha-2}_{L^{2^*}(\Sigma_{\lambda_0+\tau} \setminus K)} \int_{\Sigma_{\lambda_0+\tau} \setminus K}|\nabla \xi^+_{\lambda_0+\tau}|^2 \,dx \\
	&+ C_2 C^2_{v,S} \| \hat{v}\|^{\beta-2}_{L^{2^*}(\Sigma_{\lambda_0+\tau} \setminus K)} \int_{\Sigma_{\lambda_0+\tau} \setminus K}|\nabla \zeta^+_{\lambda_0 + \tau}|^2 \,dx,
	\end{split}
	\end{equation}
	
	where $C_1:=2 \frac{n+2}{n-2} \| \hat{u}\|^{\beta}_{L^{2^*}(\Sigma_{\lambda_0+\tau} \setminus K)} + \frac{\alpha (2^*+\beta-1)}{2^*} \| \hat{v}\|^{\beta}_{L^{2^*}(\Sigma_{\lambda_0+\tau} \setminus K)}$, $C_2:= 2 \frac{n+2}{n-2} \| \hat{v}\|^{\alpha}_{L^{2^*}(\Sigma_{\lambda_0+\tau} \setminus K)} +  \frac{\beta (2^*+\beta-1)}{2^*} \|\hat{u}\|^{\alpha}_{L^{2^*}(\Sigma_{\lambda_0+\tau} \setminus K)}$, $C_{u,S}$ and $C_{v,S}$ are the Sobolev constants. Now taking the compact set $K$ sufficiently large and thanks to \eqref{E:compact}, we can fix $\delta>0$ such that
	$$\delta < \min\{C_1  C^2_{u,S} \| \hat{u}\|^{\alpha-2}_{L^{2^*}(\Sigma_{\lambda_0+\tau} \setminus K)} ,C_2 C^2_{v,S} \| \hat{v}\|^{\beta-2}_{L^{2^*}(\Sigma_{\lambda_0+\tau} \setminus K)} \}$$ 
	and we observe that, thanks to \eqref{E:compact}, with this choice we have
	\[
	C_1  C^2_{u,S} \| \hat{u}\|^{\alpha-2}_{L^{2^*}(\Sigma_{\lambda_0+\tau} \setminus K)}< 1 \qquad \text{and} \qquad  C_2 C^2_{v,S} \| \hat{v}\|^{\beta-2}_{L^{2^*}(\Sigma_{\lambda_0+\tau} \setminus K)} <1, \qquad \forall \,\, 0 \le \tau<\bar\tau
	\]
	which plugged into \eqref{E:finaldoublicritmov} implies that
	$\displaystyle \int_{\Sigma_{\lambda_0+\tau} \setminus K}|\nabla \xi^+_{\lambda_0+\tau}|^2 \,dx = \int_{\Sigma_{\lambda_0+\tau} \setminus K}|\nabla \zeta^+_{\lambda_0+\tau}|^2 \,dx = 0$ for every $0 \le \tau<\bar\tau$.
	Hence $\displaystyle \int_{\Sigma_{\lambda_0+\tau}} |\nabla \xi^+_{\lambda_0+\tau}|^2 \,dx = \int_{\Sigma_{\lambda_0+\tau}}|\nabla \zeta^+_{\lambda_0+\tau}|^2 \,dx = 0$ for every $0 \le \tau<\bar\tau$,
	since $\nabla \xi_{\lambda_0+\tau}^+$ and $\nabla \zeta^+_{\lambda_0+\tau}$ are  zero in a neighbourhood of $K$. The latter and Lemma \ref{E:stimgrad} imply that $
	\xi_{\lambda_0+\tau}^+ =0 $ and  $
	\zeta_{\lambda_0+\tau}^+ =0 $ on $ \Sigma_{\lambda_0 +\tau}$ for every
	$0 \le \tau<\bar\tau$ and thus $\hat{u}\leq \hat{u}_{\lambda_0+\tau}$ and $\hat{v}\leq \hat{v}_{\lambda_0+\tau}$ in
	$\Sigma_{\lambda_0+\tau}\setminus R_{\lambda_0+\tau}(\Gamma^* \cup
	\{0\})$ for every $0 \le \tau<\bar\tau$ . Which proves the claim of
	Step 2.
	
	\noindent \emph{Step 3: conclusion.} The symmetry of the Kelvin
	transform $v$ follows now performing the moving plane method in the
	opposite direction. The fact that $\hat{u}$ and $\hat{v}$ are symmetric w.r.t. the
	hyperplane $\{ x_1=0 \}$ implies the symmetry of the solution $(u,v)$
	w.r.t. the hyperplane $\{ x_1=0 \}$. The last claim then follows by
	the invariance of the considered problem with respect to isometries
	(translations and rotations).
	
\end{proof}

\bigskip

\end{document}